\documentclass[final,onefignum,onetabnum]{siamart171218}

\usepackage{tikz-cd}
\usepackage{amsfonts}
\usepackage{graphicx}
\usepackage{epstopdf}
\usepackage{algorithmic}
\ifpdf
  \DeclareGraphicsExtensions{.eps,.pdf,.png,.jpg}
\else
  \DeclareGraphicsExtensions{.eps}
\fi

\newcommand{\cthree}{c_2}

\newsiamremark{remark}{Remark}
\newsiamremark{hypothesis}{Hypothesis}
\newsiamremark{problem}{Problem}
\crefname{hypothesis}{Hypothesis}{Hypotheses}
\newsiamthm{claim}{Claim}

\headers{The imaginary part of the scattering Green function}{A. D. Agaltsov, T. Hohage, and R. G. Novikov}

\title{The imaginary part of the scattering Green function: monochromatic relations to the real part and uniqueness results for inverse problems}

\author{Alexey D. Agaltsov\thanks{Max-Planck Institute for Solar Systems Research, 
Justus-von-Liebig-Weg 3, 37077 G\"ottingen, Germany 
  (\email{agaltsov@mps.mpg.de}).}
\and Thorsten Hohage\thanks{Institute for Numerical and 
Applied Mathematics, University of G\"ottingen, Lotzestr. 16-18, 37083 G\"ottingen, Germany 
  (\email{hohage@math.uni-goettingen.de}) and Max-Planck Institute for Solar Systems Research.}
\and Roman G. Novikov\thanks{ CMAP, Ecole Polytechnique, CNRS, Universit\'e Paris-Saclay, 91128 Palaiseau, France; and IEPT RAS, 117997 Moscow, Russia
  (\email{novikov@cmap.polytechnique.fr}).}
 }

\usepackage{amsopn}

\ifpdf
\hypersetup{
  pdftitle={The imaginary part of the scattering Green function},
  pdfauthor={A. D. Agaltsov, T. Hohage, and R. G. Novikov}
}
\fi

\begin{document}

\maketitle

\begin{abstract}
  For many wave propagation problems with random sources it has been demonstrated  that cross correlations of wave fields are proportional to the imaginary part of the Green function of the underlying wave equation. This leads to the inverse problem to recover coefficients of a wave equation from the imaginary part of the Green function on some measurement manifold. In this paper we prove, in particular, local uniqueness results for the Schr\"odinger equation with one frequency and for the acoustic wave equation with unknown density and sound speed and two frequencies. As the main tool of our analysis, we establish new algebraic identities between the real and the imaginary part of Green's function, which 
in contrast to the well-known Kramers-Kronig relations involve only one frequency. 
\end{abstract}



\section{Introduction}

In classical inverse scattering problems one considers a known deterministic source or incident wave 
and aims to reconstruct a scatterer (e.g.\ the inhomogeniety of a medium) from measurements of scattered waves. 
In the case of point sources this amounts to measuring the Green function of the underlying 
wave equation on some observation manifold. 
From the extensive literature on such problems we only mention 
uniqueness results in \cite{novikov:88,nachman:88,bukhgeim:08,FKSU:07,AN:15}, stability results in 
\cite{stefanov:90,HH:01,IN:13}, and the books \cite{PS:02,CK:13} 
concerning many further aspects. 

Recently there has been a growing interest in inverse wave propagation problem with \emph{random} sources. 
This includes passive imaging in seismology (\cite{snieder_etal:09}), ocean acoustics (\cite{BSS:08}),  
ultrasonics (\cite{WL:01}), and local helioseismology (\cite{gizon_etal:17}). 
It is known that in such situations cross correlations of randomly excited 
waves contain a lot of information about the medium. In particular, it has been demonstrated both theoretically 
and numerically that under certain assumptions such cross correlations are proportional to the imaginary 
part of the Green function in the frequency domain. This leads to inverse problems where some coefficient(s) 
in a wave equation have to be recovered given only the imaginary part of Green's function. The purpose 
of this paper is to prove some first uniqueness results for such inverse problems. For results on related problems in the time domain see, e.g., \cite{GP:09} and references therein.

Recall that for a random solution $u(x,t)$ of a wave equation modeled as a stationary random process, 
the empirical cross correlation function over an interval $[0,T]$ with time lag $\tau$ is defined by 
\[
C_T(x_1,x_2,\tau) := \frac{1}{T}\int_0^T u(x_1,t)u(x_2,t+\tau)\,dt, \quad \tau \in \mathbb R,
\]
In numerous papers it has been demonstrated that under certain conditions the time derivative of the 
cross correlation function is proportional to the symmetrized outgoing Green function
\begin{equation*}
\frac{\partial}{\partial \tau} \mathbb{E}\left[C_T(x_1,x_2,\tau)\right] 
\sim - [G(x_1,x_2,\tau)-G(x_1,x_2,-\tau)], \quad \tau \in \mathbb R.
\end{equation*}
Taking a Fourier transform of the last equation with respect to $\tau$ one arrives at the relation 
\begin{equation*}
\mathbb{E}\left[\widehat{C_T}(x_1,x_2,k)\right] 
\sim \frac{1}{ki}\left[G^+(x_1,x_2,k)-\overline{G^+(x_1,x_2,k)}\right]
= \frac{2}{k} \Im G^+(x_1,x_2,k), \quad k \in \mathbb R.
\end{equation*}
Generally speaking, these relations have been shown 
to hold true in situations where the energy is equipartitioned, e.g. in an open domain the recorded signals 
are a superposition of plane waves in all directions with uncorellated and identically distributed amplitudes or in a bounded domain that amplitudes of normal modes are uncorrelated and identically distributed, see 
\cite{GP:09,RSKR:05,gizon_etal:17,snieder:07,SL:13}. This condition is fulfilled if the sources are 
uncorrelated and appropriately distributed over the domain or if there is enough scattering. 

This paper has mainly been motivated by two inverse problems in local helioseismology and in ocean tomography. 
In both cases we consider the problem of recovering the density $\rho$ and the compressibility 
$\kappa$ (or equivalently the sound velocity $c=1/\sqrt{\rho\kappa}$) in the acoustic wave equation  
\begin{align}\label{in.ac}
 \nabla \cdot \big( \tfrac{1}{\rho(x)} \nabla p\big) + \omega^2 \kappa(x) p = f, \quad x \in \mathbb R^d, \; d \geq 2,
\end{align}
with random sources $f$. We assume that correlation data proportional to the imaginary part of 
Green's function for this differential equation are available on the boundary of a bounded domain 
$\Omega\subset \mathbb R^d$ for two different values of the frequency $\omega > 0$ and that $\rho$ and 
$\kappa$ are constant outside of $\Omega$. As a main result we will show that $\rho$ and $\kappa$ are 
uniquely determined by such data in some open neighborhood of any reference model 
$(\rho_0,\kappa_0)$. 

Let us first discuss the case of helioseismology in some more detail: Data on the line of sight velocity of 
the solar surface have been collected at 
high resolution for several decades by satellite based Doppler shift measurements (see \cite{GBS:10}). 
Based on these data, correlations of acoustic waves excited by turbulent convection can be computed, which 
are proportional to the imaginary part of Green's functions under assumptions mentioned above.   
These data are used to reconstruct quantities in the interior of the Sun such as sound velocity, 
density, or flows (see e.g.~\cite{HGS:16}). The aim of this paper is to contribute to the theoretical foundations of such reconstruction method by showing local uniqueness in the simplified model above. 

In the case of ocean tomography we consider measurements of correlations of ambient noise by 
hydrophones placed at the boundary of a circular area of interest. If the ocean is modeled as 
a layered separable waveguide, modes do not interact, and each horizontal mode satisfies the two-dimensional 
wave equation of the form \cref{in.ac} (see \cite{BSS:07,BSS:08}). 

The problem above can be reduced to the following simpler problem of independent interest: 
Determine the real-valued potential $v$ in the Schr\"odinger equation
\begin{align}\label{in.eq}
  -\Delta \psi + v(x) \psi = k^2 \psi, \quad x \in \mathbb R^d, \; d \geq 2,\; k > 0
\end{align}
given the imaginary part of the outgoing Green function $G^+_v(x,y,k)$ for one $k>0$ and all $x,y$ on the 
boundary of a domain containing the support of $v$. This problem is a natural fixed energy version of 
the multi-dimensional inverse spectral problem formulated originally by M.G.~Krein, I.M.~Gelfand and 
B.M.~Levitan at a conference on differential equations in Moscow
in 1952 (see \cite{berezanskii:58}).

In this connection recall that the Schr\"odinger operator admits the following spectral decomposition in $L^2(\mathbb R^d)$:
\begin{equation}
 \begin{gathered}
  -\Delta + v(x) = \int_0^\infty \lambda^2 \, dE_\lambda + \sum\nolimits_{j=1}^N E_j e_j \otimes \overline e_j, \\
  dE_\lambda = \tfrac 2 \pi \Im R^+_v(\lambda)\lambda d\lambda,
 \end{gathered}
\end{equation}
where $dE_\lambda$ is the positive part of the spectral measure for $-\Delta+v(x)$, $E_j$ are non-positive eigenvalues of $-\Delta + v(x)$ corresponding to normalized eigenfunctions $e_j$, known as bound states, and $R^+_v(\lambda) = (-\Delta+v-\lambda^2-i0)^{-1}$ is the limiting absorption resolvent for $-\Delta+v(x)$, whose Schwartz kernel is given by $G^+_v(x,y,\lambda)$,  
see, e.g., \cite[Lem.14.6.1]{hoermander:83}.

The plan of the rest of this paper is as follows: In the following section we present our main results, 
in particular algebraic relations between $\Im G^+_v$ and $\Re G^+_v$ on $\partial\Omega$ at fixed $k$ 
(\cref{re.thm.ABr} and \cref{re.thm.AtBtr}) and local uniqueness results given only the imaginary part 
of Green's function (\cref{un.thm.sch} and \cref{un.thm.hel}). 
The remainder of the paper is devoted to the proof of these results (see \cref{in.fig.sch}). After discussing the mapping 
properties of some boundary integral operators in \cref{sec:mapping} we give 
the rather elementary proof of the relations between $\Im G^+_v$ and $\Re G^+_v$ at fixed $k$ 
in \cref{sec:imreG}. By these relations, $\Re G^+_v$ is uniquely determined by $\Im G^+_v$ up to the 
signs of a finite number of certain eigenvalues. To fix these signs, we will have to derive an inertia theorem 
in \cref{sec:inertia}, before we can finally show in \cref{sec:uniqueness} that 
$\Re G^+_v$ is locally uniquely determined by $\Im G^+_v$ and appeal to known uniqueness results for 
the full Green function to complete the proof of our uniqueness theorems. 
Finally, in \cref{un.lem.genp} we discuss the assumptions of our uniqueness theorems before the paper 
ends with some conclusions. 

\begin{figure}[h]
	\begin{center}
 \begin{tikzcd}[ampersand replacement=\&, column sep=small]
	    \&[.8em] \Im \mathcal G_v \arrow[d,equal] \& T B T^* \arrow[d,equal] \&[1.4em] \&[1.3em] TAT^* \arrow[d,equal] \& \Re \mathcal G_v \arrow[d,equal] \\[-1em]
	    \Im G^+_v \arrow[r]{}{\text{\cref{in.G0}}}[swap]{\text{\cref{re.ABQ}}} \&[.6em] B \arrow[r,"\text{\cref{re.AtBt}}"] \arrow[drr,bend right=14] \&[.6em] \widetilde B \arrow[r,"\text{Thm.\ref{re.thm.AtBtr}}"] \& \text{$|\widetilde A|$} \ar[r]{}{\text{Prp.\ref{un.prp.hel}}}[swap]{\text{Prp.\ref{un.prp.sch}}} \& \widetilde A \arrow[r,"\text{$T$ inj.}"] \& A \arrow[dll,bend left=15] \\[-1em]
		\&\&\& \mathcal G_v \arrow[d,"\text{\cite{novikov:88,berezanskii:58}}"] \&\&  \\
		\&\&\& v \&\& 
 \end{tikzcd}\end{center}
 \caption{Schema of demonstration of \cref{un.thm.sch,un.thm.hel}}\label{in.fig.sch}
 \end{figure}

\section{Main results}
\label{sec:main}

For the Schr\"odinger equation \cref{in.eq} we will assume that 
 \begin{align}
\label{in.v}
  & v \in L^\infty(\Omega,\mathbb R), \qquad 
   v = 0 \quad \text{on $\mathbb R^d \setminus \Omega$},\\
	\label{in.dom}
	& \text{$\Omega$ is an open bounded domain in $\mathbb R^d$ with $\partial\Omega \in C^{2,1}$},
 \end{align}
where by definition a $\partial\Omega \in C^{2,1}$ means that $\partial\Omega$ is locally a graph of a $C^2$ function with Lipschitz continuous second derivatives, see \cite[p.90]{mclean:00} for more details.

Moreover, we suppose that
\begin{align}\label{in.Dir}
 &\text{$k^2$ is not a Dirichlet eigenvalue of $-\Delta+v(x)$ in $\Omega$}. 
\end{align}
For equation \cref{in.eq} at fixed $k>0$ we consider the outgoing Green function $G^+_v = G^+_v(x,y,k)$, which is for any $y\in {\mathbb R}^d$ the solution to the following problem:
\begin{subequations}\label{in.G+}
\begin{align}
  &(-\Delta + v  - k^2 ) G^+_v(\cdot,y,k) = \delta_y, \\
	\label{eq:Sommerfeld}
  &\bigl(\tfrac{\partial}{\partial |x|} - i k \bigr) G^+_v(x,y,k) = o\bigl( |x|^\frac{1-d}{2} \bigr), \quad |x|\to+\infty.
\end{align}
\end{subequations}
Recall that $G_v(x,y,k) = G_v(y,x,k)$ by the reciprocity principle.

In the present work we consider, in particular, the following problem:

\begin{problem}\label{in.prob.im} Determine the coefficient $v$ in the Schr\"odinger equation \cref{in.eq} from $\Im G^+_v(x,y,k)$ given at all $x$, $y \in \partial \Omega$, at fixed $k$.
\end{problem}

As discussed in the introduction, 
mathematical approaches to \cref{in.prob.im} are not yet well developed in the literature in contrast with the case of the following inverse problem from $G^+_v$ (and not only from $\Im G^+_v$):
\begin{problem}\label{in.prob.full} Determine the coefficient $v$ in the Schr\"odinger equation \cref{in.eq} from $G^+_v(x,y,k)$ given at all $x$, $y \in \partial\Omega$, at fixed $k$. 
\end{problem}

For the acoustic equation \cref{in.ac} we impose the assumptions that
\begin{subequations}
\begin{align}
\label{in.rho}
  &\rho \in W^{2,\infty}(\mathbb R^d,\mathbb R), \qquad 
  &&\rho(x) > 0, \quad x \in \Omega, \qquad 
  &&\rho(x) = \rho_c > 0, \quad  x \not\in \Omega,\\
 \label{in.kappa}
  &\kappa \in L^\infty(\Omega,\mathbb R),\qquad
  &&\kappa(x) = \kappa_c>0, \quad x \not \in \Omega
\end{align}
\end{subequations}
for some constants $\rho_c$ and $\kappa_c$. For equation \cref{in.ac} we consider the radiating Green function $P = P_{\rho,\kappa}(x,y,\omega)$, which is the solution of the following problem:
\begin{equation}\label{in.Pac}
 \begin{gathered}
  \nabla \cdot \big( \tfrac{1}{\rho} \nabla P(\cdot,y,\omega)\big) + \omega^2 \kappa P(\cdot,y,\omega) = -\delta_y, \quad \omega > 0, \\
  \bigl(\tfrac{\partial}{\partial |x|} - i \omega \sqrt{\rho_c \kappa_c} \bigr) P(x,y,\omega) = o\bigl( |x|^\frac{1-d}{2} \bigr), \quad |x|\to+\infty.
 \end{gathered} 
\end{equation}

In the present work we consider the following problem for equation \cref{in.ac}:
\begin{problem}\label{in.prob.ac} Determine the coefficients $\rho$, $\kappa$ in the acoustic equation \cref{in.ac} from $\Im P_{\rho,\kappa}(x,y,\omega)$ given at all $x$, $y \in\partial\Omega$, and for a finite number of $\omega$.
\end{problem}

\paragraph*{Notation} If  $X$ and $Y$ are Banach spaces, we will denote the space of bounded linear operators 
from $X$ to $Y$ by $L(X,Y)$ and write $L(X):=L(X,X)$. 
Moreover, we will denote the subspace of compact operators in $L(X,Y)$ by  $K(X,Y)$, and the subset 
of operators with a bounded inverse by $GL(X,Y)$. 

Besides, we denote by $\|\cdot\|_\infty$ the norm in $L^\infty(\Omega)$, and by $\langle\cdot,\cdot\rangle$, 
$\|\cdot\|_2$ the scalar product and the norm in $L^2(\partial\Omega)$. Furthermore, 
we use the standard notation $H^s(\partial\Omega)$ for $L^2$-based Sobolev spaces of index 
$s$ on $\partial \Omega$ (under the regularity assumption \cref{in.dom} we need $|s|\leq 3$).

In addition, the adjoint of an operator $A$ is denoted by $A^*$.

\subsection{Relations between $\Re {\mathcal G}$ and $\Im {\mathcal G}$}

For fixed $k>0$ let us introduce the integral operator
${\mathcal G}_v(k)\in L(L^2(\partial\Omega))$ by 
\begin{equation}\label{in.G0}
({\mathcal G}_v(k) \varphi)(x) := \int_{\partial\Omega} G^+_v(x,y,k) \varphi(y)\,ds(y),\qquad x\in\partial\Omega
\end{equation}
where $ds(y)$ is the hypersurface measure on $\partial \Omega$. For the basic properties of $\mathcal G_v(k)$ see, e.g., \cite[Chapter 7]{mclean:00}. Note that for the case $v=0$ 
the Green function $G^+_0$ is the outgoing fundamental solution to the Helmholtz equation, and 
${\mathcal G}_0$ is the corresponding single layer potential operator.

Recall that 
\begin{equation*}
  \sqrt{\sigma_\Omega(-\Delta)} := \bigl\{ k>0 \colon \text{$k^2$ is a Dirichlet eigenvalue of $-\Delta$ in $\Omega$} \bigr\}
\end{equation*}
is a discrete subset of $(0,+\infty)$ without accumulation points.

\begin{theorem}\label{re.thm.ABr} Suppose that  $\Omega$, $k$, and $v$ satisfy the conditions \cref{in.v}, 
\cref{in.dom}, \cref{in.Dir}. Then:
\begin{enumerate}
 \item The mapping 
\begin{equation}\label{re.Q}
 (0,+\infty) \setminus \sqrt{\sigma_\Omega(-\Delta)} \to L(H^1(\partial\Omega),L^2(\partial\Omega)), \quad \lambda \mapsto Q(\lambda) := \Im {\mathcal G}_0^{-1}(\lambda)
\end{equation}
has a unique continuous extension to $(0,+\infty)$. In following we will often write $Q$ instead of $Q(k)$.
 \item ${\mathcal G}_v(k)\in L(L^2(\partial\Omega),H^1(\partial\Omega))$ and the operators
\begin{equation}\label{re.ABQ}
  A := \Re {\mathcal G}_v(k), \quad B := \Im {\mathcal G}_v(k)
\end{equation}
satisfy the following relations:
\begin{subequations}\label{re.ABr}
\begin{align}
  & AQA + BQB = -B \label{re.ABr.1}\\
  & AQB - BQA = 0. \label{re.ABr.2}
\end{align}
\end{subequations}
\end{enumerate}

\end{theorem}
\Cref{re.thm.ABr} is proved in \cref{re.thm.ABrp}.

We would like to emphasize that relations \cref{re.ABr.1}, \cref{re.ABr.2} are valid in any dimension $d \geq 1$.

For the next theorem, recall that the exterior boundary value problem 
\begin{subequations}\label{eqs:extBVP}
\begin{align}
&\Delta u + k^2 u = 0&& \mbox{in }\mathbb{R}^d\setminus\overline{\Omega},\\
&u = u_0 && \mbox{on }\partial\Omega,\\ 
\label{eq:SRC}
&\tfrac{\partial u}{\partial |x|}-ik u = o\bigl( |x|^{(1-d)/2} \bigr)&&\mbox{as } |x|\to\infty
\end{align}
\end{subequations}
has a unique solution for all $u_0 \in C(\partial \Omega)$, which has the asymptotic behavior 
\[
u(x) = 
\frac{e^{ik|x|}}{|x|^{(d-1)/2}}u_{\infty}\left(\frac{x}{|x|}\right)
\left(1+O\left(\frac{1}{|x|}\right)\right),\quad |x|\to \infty.
\]
Here $u_\infty\in L^2(S^{d-1})$ is called the \emph{farfield pattern} of $u$. 

\begin{theorem}\label{re.thm.AtBtr}  Suppose that  $\Omega$, $k$, and $v$ satisfy the conditions \cref{in.v}, 
\cref{in.dom}, and \cref{in.Dir}. Then:
\begin{enumerate}
 \item[I.] The operator $C(\partial\Omega)\to L^2(S^{d-1})$, $u_0\mapsto \sqrt{k}u_{\infty}$ mapping 
Dirichlet boundary values $u_0$ to the scaled farfield pattern $u_\infty$ of the solution to 
\eqref{eqs:extBVP} has a continuous extension to an operator 
$T(k) \in L\bigl( L^2(\partial\Omega),L^2(S^{d-1}))$, and $T(k)$ is compact, injective, and has dense range. 
Moreover,  $Q(k)$ defined in \cref{re.thm.ABr} has a continuous extension to $L(L^2(\partial\Omega))$ satisfying
\begin{align}\label{eq:Q_factorization}
Q(k) = - T^*(k) T(k).
\end{align}
 \item[II.] The operators  $\widetilde A$, $\widetilde B \in L\bigl(L^2(S^{d-1})\bigr)$  defined by 
\begin{equation}\label{re.AtBt}
  \widetilde A := \Re \widetilde {\mathcal G}_v(k), \quad \widetilde B := \Im \widetilde {\mathcal G}_v(k), 
	\quad \widetilde{\mathcal G}_v(k) := T(k){\mathcal G}_v(k)T^*(k)
\end{equation}
are compact and symmetric and satisfy the relations 
\begin{subequations}\label{re.AtBtr}
\begin{align}
 & \widetilde A^2 = \widetilde B - \widetilde B^2, \label{re.AtBtr.1}\\
 & \widetilde A \widetilde B = \widetilde B \widetilde A. \label{re.AtBtr.2}
\end{align}
\end{subequations}

 \item[III.] The operators $\widetilde A$, $\widetilde B$ are simultaneously diagonalisable in $L^2(S^{d-1})$.  
Moreover, if $\widetilde{\mathcal G}_v(k)f = ( \lambda_{\widetilde A} +i\lambda_{\widetilde B})f$ for 
some $f\neq 0$ and $\lambda_{\widetilde A},\lambda_{\widetilde B}\in\mathbb{R}$, then 
\begin{equation}\label{re.lAlBr}
  \lambda_{\widetilde A}^2 = \lambda_{\widetilde B} -\lambda_{\widetilde B}^2.
\end{equation}

\end{enumerate}
\end{theorem}
\Cref{re.thm.AtBtr} is proved in \cref{re.thm.AtBtrp}.

We could replace $T(k)$ by any operator satisfying \eqref{eq:Q_factorization} in most of this paper, 
e.g. $\sqrt{-Q(k)}$. However,  $\widetilde{\mathcal G}^*_v(k)$ has a physical interpretation given in 
Lemma \ref{ge.lem.sct}, and this will be used to verify condition \eqref{un.Geig} below.

In analogy to the relations \cref{re.ABr.1} and \cref{re.ABr.2}, the relations \cref{re.AtBtr.1} and 
\cref{re.AtBtr.2} are also valid in any dimension $d \geq 1$.

\begin{remark}\label{re.rem.KK} The algebraic relations between $\Re G^+_v$ and $\Im G^+_v$ given in \cref{re.thm.ABr} and \cref{re.thm.AtBtr} involve only one frequency in contrast to well-known Kramers-Kronig relations which under certain conditions are as follows:
\begin{align*}
    \Re G^+_v(x,y,k) & = \frac 1 \pi \mathop{\mathrm{p.v.}}\int_{-\infty}^{+\infty} \frac{\Im G^+_v(x,y,k')}{k'-k} dk', \\
    \Im G^+_v(x,y,k) &= -\frac 1 \pi \mathop{\mathrm{p.v.}} \int_{-\infty}^{+\infty} \frac{\Re G^+_v(x,y,k')}{k'-k} dk',
\end{align*}
where $x \neq y$, $k\in \mathbb R$ for $d=3$ or $k \in \mathbb R\setminus\{0\}$ for $d=2$, and $G^+_v(x,y,-k) := \overline G^+_v(x,y,k)$, $G^+_0(x,y,-k) := \overline G^+_0(x,y,k)$, $k > 0$.

In this simplest form the Kramers-Kronig relations are valid, for example, for the Schr\"odinger equation \cref{in.eq} under conditions \cref{in.v}, \cref{in.dom}, $d = 2$, $3$, if the discrete spectrum of $-\Delta + v$ in $L^2(\mathbb R^d)$ is empty and $0$ is not a resonance (that is, a pole of the meromorphic continuation of the resolvent $k\mapsto(-\Delta+v-k^2)^{-1}$).
\end{remark}

\subsection{Identifiability of $v$ from $\Im {\mathcal G}_v$}

We suppose that
\begin{equation}\label{un.Geig}
	\begin{gathered}
	  \text{if $\lambda_1$, $\lambda_2$ are eigenvalues of $\widetilde{\mathcal G}_v(k)$ with $\Im \lambda_1 = \Im \lambda_2$, then $\Re \lambda_1 = \Re \lambda_2$},
	\end{gathered}
\end{equation}
where $\widetilde{\mathcal G}_v(k) = \widetilde A + i \widetilde B$ is the operator defined in \cref{re.thm.AtBtr}. Under this assumption any eigenbasis of $\widetilde B$ in $L^2(S^{d-1})$ is also an eigenbasis for $\widetilde A$ in $L^2(S^{d-1})$ in view of \cref{re.thm.AtBtr} (III).

\begin{theorem}\label{un.thm.sch} Let $\Omega$ satisfy \cref{in.dom}, $d \geq 2$, $v_0$ satisfy \cref{in.v} and let $k>0$ be such that $\Re{\mathcal G}_{v_0}(k)$ is injective in $H^{-\frac 1 2}(\partial \Omega)$ and \cref{un.Geig} holds true with $v=v_0$. Then there exists 
$\delta = \delta(\Omega,k,v_0) > 0$ such that for any $v_1$, $v_2$ satisfying \cref{in.v} and 
\[
  \|v_1-v_0\|_\infty \leq \delta, \quad \|v_2-v_0\|_\infty \leq \delta,
\]
the equality $\Im G^+_{v_1}(x,y,k) = \Im G^+_{v_2}(x,y,k)$ for all $x, y \in \partial\Omega$ implies that $v_1 = v_2$.
\end{theorem}

\Cref{un.thm.sch} is proved in \cref{un.thm.schp}. 
In \cref{un.lem.genp} we present results indicating that the assumptions of this theorem are ``generically'' satisfied.

We also mention the following simpler uniqueness result for $\Re G^+_v$ based on analytic continuation if 
$\Im G^+_v$ is given not only for one frequency, but for an interval of frequencies. This uniqueness result 
is even global. However, analytic continuation is notoriously unstable, and computing 
$\Im G^+_v$ on an interval of frequencies from time dependent data would require an infinite time window. 
Therefore, it is preferable to work with a discrete set of frequencies.

\begin{proposition}\label{un.thm.int} Let $\Omega$ satisfy \cref{in.dom}, $d \in \{2,3\}$, and $v_1$, $v_2$ satisfy \cref{in.v}. Suppose that the discrete spectrum of the operators $-\Delta + v_j$ in $L^2(\mathbb R^d)$ is empty and $0$ is not a resonance (that is, a pole of the meromorphic continuation of the resolvent $R^+_{v_j}(k)=(-\Delta+v_j-k^2-i0)^{-1}$), $j = 1$, $2$. Besides, let $x$, $y \in \mathbb R^d$, $x \neq y$, be fixed. Then if $\Im G^+_{v_1}(x,y,k) = \Im G^+_{v_2}(x,y,k)$ for all $k \in (k_0-\varepsilon,k_0+\varepsilon)$ for some fixed $k_0 > 0$, $\varepsilon > 0$, then $G^+_{v_1}(x,y,k) = G^+_{v_2}(x,y,k)$ for all $k>0$. In addition, if $\Im G^+_{v_1}(x,y,k) = \Im G^+_{v_2}(x,y,k)$ for all $x$, $y \in \partial\Omega$, $k \in (k_0-\varepsilon,k_0+\varepsilon)$, then $v_1 = v_2$.
\end{proposition}
\begin{proof} Under the assumptions of \cref{un.thm.int} the functions $G^+_{v_j}(x,y,k)$ at fixed $x \neq y$ admit analytic continuation to a neighborhood of each $k \in \mathbb R$ ($k \neq 0$ for $d=2$) in $\mathbb C$. It follows that $\Im G^+_{v_j}(x,y,k)$ are real-analytic functions of $k \in \mathbb R$ ($k \neq 0$ for $d=2$). 
Moreover, $\Im G^+_{v_j}(x,y,-k) = -\Im G^+_{v_j}(x,y,k)$ for all $k > 0$. Hence, the equality $\Im G^+_{v_1}(x,y,k) = \Im G^+_{v_2}(x,y,k)$ for $k\in(k_0-\varepsilon,k_0+\varepsilon)$ implies the same equality for all $k \in \mathbb R$ ($k \neq 0$ for $d=2$). Taking into account Kramers-Kronig relations recalled in \cref{re.rem.KK}, we obtain, in particular, that $G^+_{v_1}(x,y,k) = G^+_{v_2}(x,y,k)$, $k > 0$.

Moreover, the equality $G^+_{v_1}(x,y,k)=G^+_{v_2}(x,y,k)$, $x$, $y \in \partial\Omega$, $k>0$, implies $v_1 = v_2$ see, e.g., \cite{berezanskii:58,novikov:88}.
\end{proof}

\subsection{Identifiability of $\rho$ and $\kappa$ from $\Im P_{\rho,\kappa}$}

Let $P_{\rho,\kappa}(x,y,\omega)$ be the function of \cref{in.Pac} and define $\mathcal P_{\rho,\kappa,\omega}$, $\widetilde{\mathcal P}_{\rho,\kappa,\omega}$ as
\begin{gather*}
  \bigl(\mathcal P_{\rho,\kappa,\omega} u \bigr)(x) := \int_{\partial \Omega} P_{\rho,\kappa}(x,y,\omega)u(y) \, ds(y), \quad x \in \partial \Omega, \; u \in H^{-\frac 1 2}(\partial \Omega), \\
  \widetilde{\mathcal P}_{\rho,\kappa,\omega} := T(k)\mathcal P_{\rho,\kappa,\omega}T^*(k), \quad k := \omega\sqrt{\rho_c\kappa_c},
\end{gather*}
where $T(k)$ is the same as in \cref{re.thm.AtBtr}. We suppose that
\begin{equation}\label{ac.Peig}
	\text{if $\lambda_1$, $\lambda_2$ are eigenvalues of $\widetilde{\mathcal P}_{\rho,\kappa,\omega}$ with $\Im \lambda_1 = \Im \lambda_2$, then $\Re \lambda_1 = \Re \lambda_2$}.
\end{equation}
Let $W^{2,\infty}(\Omega)$ denote the $L^\infty$-based Sobolev space of index $2$.

The following theorems are local uniqueness results for the acoustic equation \cref{in.ac}.

\begin{theorem}\label{ac.thm.two} Let $\Omega$ satisfy \cref{in.dom}, $d \geq 2$, and suppose that $\rho_0$, $\kappa_0$ satisfy \cref{in.rho}, \cref{in.kappa} for some known $\rho_c$, $\kappa_c$. Let $\omega_1$, $\omega_2$ be such $\Re \mathcal P_{\rho_0,\kappa_0,\omega_j}$ is injective in $H^{-\frac 1 2}(\partial \Omega)$ and \cref{ac.Peig} holds true with $\rho=\rho_0$, $\kappa=\kappa_0$, $\omega=\omega_j$, $j = 1$, $2$. Besides, let $\rho_1$, $\kappa_1$ and $\rho_2$, $\kappa_2$ be two pairs of functions satisfying \cref{in.rho}, \cref{in.kappa}. Then there exist constants $\delta_{1,2} = \delta_{1,2}(\Omega,\omega_1,\omega_2,\kappa_0,\rho_0)$ such that if
 \begin{align*}
&\|\rho_1-\rho_0\|_{W^{2,\infty}} \leq \delta_1, &&\|\rho_2-\rho_0\|_{W^{2,\infty}} \leq \delta_1,\\
&\|\kappa_1-\kappa_0\|_\infty \leq \delta_2, &&\|\kappa_2-\kappa_0\|_\infty \leq \delta_2,
 \end{align*}
then the equality $\Im P_{\rho_1,\kappa_1}(x,y,\omega_j) = \Im P_{\rho_2,\kappa_2}(x,y,\omega_j)$ for all $x, y \in \partial \Omega$ and 
$j\in\{1,2\}$ implies that $\rho_1 = \rho_2$ and $\kappa_1 = \kappa_2$.
\end{theorem}

\begin{proof}[Proof of \cref{ac.thm.two}] Put 
\begin{equation*}
  v_j(x,\omega) = \rho_j^{\frac 1 2}(x) \Delta \rho_j^{-\frac 1 2}(x) + \omega^2 (\rho_c\kappa_c - \kappa_j(x) \rho_j(x)), \quad k^2 = \omega^2 \rho_c \kappa_c.
\end{equation*}
Then $P_{\rho_j,\kappa_j}(x,y,\omega) = \rho_c G^+_{v_j}(x,y,k)$, where $G^+_{v_j}$ denotes the Green function for equation \cref{in.eq} defined according to \cref{in.G+}. By assumptions we obtain that
\begin{equation*}
  \Im G^+_{v_1}(x,y,k) = \Im G^+_{v_2}(x,y,k), \quad x, y \in\partial\Omega, \; k = k_1,k_2, \; k_j = \omega_j \sqrt{\rho_c\kappa_c}.
\end{equation*}
Using \cref{un.thm.sch}, we obtain that
\begin{equation*}
 v_1(x,\omega_j) = v_2(x,\omega_j), \quad x \in \Omega, \; j = 1,2.
\end{equation*}
Together with the definition of $v_j$ it follows that $\rho_1^{\frac 1 2}\Delta \rho_1^{-\frac 1 2} = \rho_2^{\frac 1 2}\Delta \rho_2^{-\frac 1 2}$ and $\kappa_1 = \kappa_2$. In turn, the equality $\rho_1^{\frac 1 2}\Delta \rho_1^{-\frac 1 2} = \rho_2^{\frac 1 2}\Delta \rho_2^{-\frac 1 2}$ together with the boundary conditions $\rho_1|_{\partial \Omega} = \rho_2|_{\partial\Omega} = \rho_c$ imply that $\rho_1 = \rho_2$, see, e.g., \cite{AN:15}. 
\end{proof}

\begin{theorem}\label{ac.thm.one} Let $\Omega$ satisfy \cref{in.dom}, $d \geq 2$, and suppose that $\rho_0$, $\kappa_0$ satisfy \cref{in.rho}, \cref{in.kappa} for some known $\rho_c$, $\kappa_c$. Let $\omega$ be such that $\Re \mathcal P_{\rho_0,\kappa_0,\omega}$ is injective in $H^{-\frac 1 2}(\partial \Omega)$ and \cref{ac.Peig} holds true with $\rho=\rho_0$, $\kappa=\kappa_0$. Besides, let $\kappa_1$, $\kappa_2$ satisfy \cref{in.kappa}. Then there exists $\delta = \delta(\Omega,\omega,\kappa_0,\rho_0)$ such that the bounds 
\begin{equation*}
 \begin{gathered}
\|\kappa_1-\kappa_0\|_\infty < \delta, \quad \|\kappa_2-\kappa_0\|_\infty < \delta,
 \end{gathered} 
\end{equation*}
and the equality $\Im P_{\rho_0,\kappa_1}(x,y,\omega) = \Im P_{\rho_0,\kappa_2}(x,y,\omega)$ for all $x, y \in \partial \Omega$ imply that $\kappa_1 = \kappa_2$.
\end{theorem}
\begin{proof}[Proof of \cref{ac.thm.one}] In analogy to the proof of \cref{ac.thm.two}, put
\begin{equation*}
  v_j(x,\omega) = \rho_0^{\frac 1 2}(x) \Delta \rho_0^{-\frac 1 2}(x) + \omega^2 (\rho_c\kappa_c - \kappa_j(x) \rho_0(x)).
\end{equation*}
Then $P_{\rho_0,\kappa_j}(x,y,\omega) = \rho_c G^+_{v_j}(x,y,k)$, where $G^+_{v_j}$ denotes the Green function for equation \cref{in.eq} defined according to \cref{in.G+}. By assumptions we obtain that 
\begin{equation*}
  \Im G^+_{v_1}(x,y,k) = \Im G^+_{v_2}(x,y,k), \quad x,y \in \partial\Omega.
\end{equation*}
Using \cref{un.thm.sch} we obtain that
\begin{equation*}
  v_1(x,\omega) = v_2(x,\omega), \quad x \in \Omega.
\end{equation*}
Now it follows from the definition of $v_j$  that $\kappa_1 = \kappa_2$. 
\end{proof}

The following uniqueness theorem for the coefficient $\kappa$ only does not require smallness of 
this coefficient, but only smallness of the frequency $\omega$. Note that it is not an 
immediate corollary to \cref{un.thm.sch} since the constant $\delta$ in \cref{un.thm.sch} depends 
on $k$. 

\begin{theorem}\label{un.thm.hel} Let $\Omega$ satisfy \cref{in.dom}, $d \geq 2$, and assume that $\rho\equiv 1$ and $\kappa_c=1$ so that \cref{in.ac} 
reduces to the Helmholtz equation 
\[
\Delta p +\omega^2 \kappa(x) p = f.
\]
Moreover, suppose that $\kappa_1$ and $\kappa_2$ are two functions satisfying \cref{in.kappa} and
\begin{equation*}
  \|\kappa_1\|_\infty \leq M, \quad \|\kappa_2\|_\infty \leq M
\end{equation*}
for some $M>0$. Then for there exists $\omega_0 = \omega_0(\Omega,M) > 0$ such that if 
$\Im P_{1,\kappa_1}(x,y,\omega) = \Im P_{1,\kappa_2}(x,y,\omega)$ for all $x$, $y \in \partial \Omega$, for some fixed $\omega\in (0,\omega_0]$, then $\kappa_1 = \kappa_2$.
\end{theorem}

\Cref{un.thm.hel} is proved in \cref{sec:uniqueness}.

\section{Mapping properties of some boundary integral operators}\label{sec:mapping}

In what follows we use the following notation:
\begin{equation*}
  \bigl(R^+_v(k) f\bigr)(x) = \int_{\Omega} G^+_v(x,y,k) f(y) \, dy, \quad x \in \Omega, \; k > 0.
\end{equation*}
\begin{remark}\label{mp.rem.R} The operator $R^+_v(k)$ is the restriction from $\mathbb R^d$ to $\Omega$ of the outgoing (limiting absorption) resolvent $k\mapsto(-\Delta + v -k^2 - i0)^{-1}$. It is known that if $v$ satisfies \eqref{in.v} and $k^2$ is not an embedded eigenvalue of $-\Delta + v(x)$ in $L^2(\mathbb R^d)$, then $R^+_v(k) \in L\bigl( L^2(\Omega), H^2(\Omega) \bigr)$, see, e.g., \cite[Thm.4.2]{agmon:75}. In turn, it is known that for the operator $-\Delta + v(x)$ with $v$ satisfying \eqref{in.v} there are no embedded eigenvalues, see \cite[Thm.14.5.5 \& 14.7.2]{hoermander:83}.
\end{remark}
Recall that the free radiating Green's function is given in terms of the 
Hankel functions $H^{(1)}_\nu$ of the first kind of order $\nu$ by
\begin{equation}\label{eq:G0}
  G^+_0(x,y,k) = \tfrac i 4 \big( \tfrac{k}{2\pi |x-y|} \big)^\nu H^{(1)}_\nu (k|x-y|)
\quad\mbox{with}\quad \nu := \tfrac d 2 - 1.
\end{equation}
In addition, we denote the single layer potential operator for the Laplace equation by 
\begin{equation}\label{mp.E}
\begin{aligned}
  &(\mathcal E f)(x) := \int_{\partial\Omega} E(x-y) f(y) \, ds(y), \quad x \in \partial\Omega \quad
	\mbox{with}\\
  &E(x-y) := 
  \begin{cases}
    - \tfrac{1}{2\pi} \ln|x-y|, & d = 2,\\
    \tfrac{1}{d(d-2)\omega_d} |x-y|^{2-d}, & d \geq 3,
  \end{cases}
\end{aligned}
\end{equation}
where $\omega_d$ is the volume of the unit $d$-ball and $E$ is the fundamental solution for the Laplace equation in $\mathbb R^d$. Note that $-\Delta_x E(x-y) = \delta_y(x)$.

\begin{lemma}\label{mp.lem.Rest} Let $v$, $v_0 \in L^\infty(\Omega,\mathbb R)$ and let $k>0$ be fixed. There exist $c_1 = c_1(\Omega,k,v_0)$, $\delta_1 = \delta_1(\Omega,k,v_0)$ such that if $\|v-v_0\|_\infty \leq \delta_1$, then
\begin{equation*}
  \| R_v^+(k) f \|_{H^2(\Omega)} \leq c_1(\Omega,k,v_0) \|f\|_{L^2(\Omega)}, \quad \text{for any $f\in L^2(\Omega)$}.
\end{equation*}
In addition, for any $M > 0$ there exist constants $c_1' = c_1'(\Omega,M)$ and $k_1 = k_1(\Omega,M)$ such that if $\|v\|_\infty \leq M$, then
\begin{equation*}
   \|R^+_{k^2 v}(k)f\|_{H^2(\Omega)}  \leq 
	\begin{cases}
	c_1'(\Omega,M)\|f\|_{L^2(\Omega)},& d\geq 3, \\
  |\ln k| \, c_1'(\Omega,M) \|f\|_{L^2(\Omega)},& d = 2
 \end{cases}
\end{equation*}
for all $f \in L^2(\Omega)$ and all $k\in (0,k_1)$.
\end{lemma}
\begin{proof} We begin by proving the first statement of the lemma. The operators $R^+_v(k)$ and $R^+_{v_0}(k)$ are related by a resolvent identity in $L^2(\Omega)$:
\begin{equation}\label{mp.res_id.1}
  R^+_v(k) = R^+_{v_0}(k) \bigl( \mathrm{Id} + (v - v_0) R^+_{v_0}(k) \bigr)^{-1}, 
\end{equation}
see, e.g., \cite[p.248]{hoermander:83} for a proof. The resolvent identity is valid, in particular, if $\|v-v_0\|_\infty < \| R^+_{v_0}(k) \|^{-1}$, where the norm is taken in $L\bigl( L^2(\Omega), H^2(\Omega) \bigr)$. It follows from \cref{mp.res_id.1} that
\begin{equation*}
  \|R^+_v(k)\| \leq \frac{\|R^+_{v_0}(k)\|}{1 - \|v-v_0\|_\infty \|R^+_{v_0}(k)\|},
\end{equation*}
where the norms are taken in $L\bigl( L^2(\Omega), H^2(\Omega) \bigr)$. Taking $\delta_1 < \|R^+_{v_0}(k)\|^{-1}$, we get the first statement of the lemma.

To prove the second statement of the lemma, we begin with the case of $v=0$. The Schwartz kernel of $R^+_0(k)$ is given by the outgoing Green function for the Helmholtz equation defined in formula \cref{eq:G0}. In particular, $\Im G^+_0(x-y,k)$ satisfies
\begin{equation}\label{re.ImG+asm}
 \Im G^+_0(x,y,k) = \frac 1 4 (2\pi)^{-\nu} k^\nu |x-y|^{-\nu} J_\nu (k|x-y|) = \frac{k^{2\nu}}{2^{2\nu+2}\pi^\nu \nu!}\bigl(1 + O(z^2)\bigr) 
\end{equation}
with the Bessel function $J_\nu = \Re H^{(1)}_\nu$ of order $\nu$, where $z = k|x-y|$ and $O$ is an entire function with $O(0)=0$. In addition,
\begin{equation}\label{mp.ReG+asm}
\Re G^+_0(x-y,k)  = \begin{cases}
E(x-y) - \tfrac{1}{2\pi}\bigl( \ln \tfrac k 2 + \gamma \bigr) ( 1 + O_2(z^2) ) + \widetilde{O}_2(z^2),&d = 2, \\
E(x-y) \bigl(1 + O_3(z^2) \bigr), & d\geq 3 \mbox{ odd}\\
E(x-y) \bigl(1 + O_d(z^2) \bigr)& \\
 \qquad - \frac{k^{d-2}}{2^{2\nu+3} \pi^{\nu+1} \nu!} \ln(z/2) (1 + \widetilde{O}_d(z^2) ),& d\geq 4 \mbox{ even}
\end{cases}
\end{equation}
where $z = k|x-y|$ and $O_d$ and $\widetilde{O}_d$  are entire functions with $O_d(0)=0=\widetilde{O}_d(0)$, 
and $\gamma$ is the Euler-Mascheroni constant, see, e.g., \cite[p.279]{mclean:00}. 
These formulas imply the second statement of the present lemma for $v = 0$.

Using the resolvent identity \cref{mp.res_id.1} we obtain
\begin{equation*}
  \|R_{k^2 v}^+(k)\| \leq \frac{\|R_0^+(k)\|}{1 - k^2 M \|R_0^+(k)\|},
\end{equation*}
if $k^2 M \|R_0^+(k)\| < 1$, where the norms are taken in $L\bigl( L^2(\Omega),H^2(\Omega) \bigr)$. This inequality together with the second statement of the lemma for $v = 0$ imply the second statement of the lemma for general $v$.
\end{proof}

\begin{lemma}\label{mp.lem.S1S2est} Let $v_0$, $v_1$, $v_2 \in L^\infty(\Omega,\mathbb R)$. Then for any $k > 0$
\begin{equation}\label{mp.S1S2cmp}
  {\mathcal G}_{v_1}(k) - {\mathcal G}_{v_2}(k) \in L\bigl( H^{-\frac 3 2}(\partial \Omega), H^{\frac 3 2}(\partial \Omega) \bigr).
\end{equation}
In addition, there exist $c_2(\Omega,k,v_0)$, $\delta_2(\Omega,k,v_0)$ such that if $\|v_1-v_0\|_\infty \leq \delta_2$, $\|v_2 - v_0\|_\infty \leq \delta_2$, then
\begin{equation}
  \|{\mathcal G}_{v_1}(k)-{\mathcal G}_{v_2}(k)\| \leq c_2(\Omega,k,v_0)\|v_1-v_2\|_\infty,
\end{equation}
where the norm is taken in $L\bigl( H^{-\frac 3 2}(\partial \Omega), H^{\frac 3 2}(\partial \Omega) \bigr)$.

Furthermore, for any $M>0$ there exist constants $c_2' = c_2'(\Omega,M)$ and $k_2 = k_2(\Omega,M)$ such that if $\|v_1\|_\infty \leq M$, $\|v_2\|_\infty \leq M$, then
\begin{equation}\label{mp.S1S2est}
  \|{\mathcal G}_{k^2 v_1}(k) - {\mathcal G}_{k^2 v_2}(k)\|  \leq 
	\begin{cases} c_2'(\Omega,M) k^2 \| v_1 - v_2 \|_\infty, & \mbox{for } d \geq 3, \\
 c_2'(\Omega,M) k^2 |\ln k|^2   \| v_1 - v_2 \|_\infty, & \mbox{for } d = 2
\end{cases}
\end{equation}
holds true for all $k\in (0,k_2)$, where the norms are taken in $L\bigl( H^{-\frac 3 2}(\partial \Omega), H^{\frac 3 2}(\partial \Omega) \bigr)$.
\end{lemma}
\begin{proof} Note that
\begin{equation}\label{mp.GviaR}
  {\mathcal G}_{v_j}(k) = \gamma R_{v_j}^+(k) \gamma^*, \quad j = 1,2,
\end{equation}
where
\begin{align*}
 \gamma & \in L\left( H^s(\Omega), H^{s-\frac 1 2}(\partial \Omega)\right) \quad \mbox{and}\quad 
 \gamma^* \in L \left( H^{-s+\frac 1 2}(\partial \Omega), \widetilde H^{-s}(\Omega) \right)
\end{align*}
for $s \in (\tfrac 1 2,2]$
are the trace map and its dual (see, e.g., \cite[Thm.3.37]{mclean:00}). Here $H^s(\Omega)$ denotes the space of distributions $u$ on $\Omega$, which are the restriction of some 
$U\in H^s({\mathbb R}^d)$ to $\Omega$, i.e.\ $u=U|_{\Omega}$, whereas $\widetilde{H}^s(\Omega)$ denotes the closure of 
the space of distributions on $\Omega$ in $H^s({\mathbb R}^d)$ (see \cite{mclean:00}). Recall that 
for a Lipschitz domain we have $H^s(\Omega)^*=\widetilde{H}^{-s}(\Omega)$ for all $s\in{\mathbb R}$ 
(\cite[Thm.\ 3.30]{mclean:00}). 

The operators $R^+_{v_1}(k)$ and $R^+_{v_2}(k)$ are subject to the resolvent identity
\begin{equation}\label{mp.res_id.3}
  R^+_{v_2}(k) - R^+_{v_1}(k) = R^+_{v_1}(k)\bigl(v_1 - v_2\bigr)R^+_{v_2}(k),
\end{equation}
see, e.g., \cite[p.248]{hoermander:83} for the proof. Together with \cref{mp.res_id.3} we obtain that 
\begin{equation}\label{mp.S1S2R1R2}
  {\mathcal G}_{v_2}(k) - {\mathcal G}_{v_1}(k) = \gamma R^+_{v_1}(k)(v_1-v_2)R^+_{v_2}(k)\gamma^*.
\end{equation}
It follows from \cref{mp.rem.R} and from a duality argument that $R_v^+(k)$ is bounded in $L\bigl(L^2(\Omega),H^2(\Omega) \bigr)$ and in $L\bigl( H^{-2}(\Omega),L^2(\Omega) \bigr)$. Taking into account that all 
maps in the sequence
\begin{gather*}
  H^{-\frac 3 2}(\partial \Omega) \overset{\gamma^*}{\longrightarrow} \widetilde H^{-2}(\Omega) \overset{R^+_{v_2}(k)}{\longrightarrow} L^2(\Omega) \overset{v_1-v_2}{\longrightarrow} L^2(\Omega) 
  \overset{R^+_{v_1}(k)}{\longrightarrow} H^2(\Omega) \overset{\gamma}{\longrightarrow} H^{\frac 3 2}(\partial \Omega)
\end{gather*}
are continuous, we get \cref{mp.S1S2cmp}.

It follows from \cref{mp.S1S2R1R2} that there exists $c_2'' = c_2''(\Omega)$ such that
\begin{equation*}
  \|{\mathcal G}_{v_1}(k) - {\mathcal G}_{v_2}(k)\| \leq c_2''(\Omega) \|R^+_{v_1}(k)\| \|R^+_{v_2}(k)\| \|v_1-v_2\|_\infty,
\end{equation*}
where the norm on the left is taken in $L\bigl( H^{-\frac 3 2}(\partial \Omega), H^{\frac 3 2}(\partial \Omega) \bigr)$, and the norms on the right are taken in $L\bigl( L^2(\Omega),H^2(\Omega) \bigr)$. Using this estimate and \cref{mp.lem.Rest} we obtain the second assertion of the present lemma. Using the estimate for a pair of potentials $(k^2 v_1, k^2 v_2)$ instead of $(v_1, v_2)$ and using \cref{mp.lem.Rest} we obtain the third assertion of the present lemma.
\end{proof}

\begin{lemma}\label{re.lem.Fred} Suppose that \cref{in.dom} holds true and $v$ satisfies \cref{in.v}. Then ${\mathcal G}_v(k)$and $\Re {\mathcal G}_v(k)$ are Fredholm operators of index zero in spaces $L(H^{s-\frac 1 2}(\partial\Omega),H^{s+\frac 1 2}(\partial\Omega))$, $s \in [-1,1]$, real analytic in $k \in (0,+\infty)$. If, in addition, $v \in C^\infty(\mathbb R^d,\mathbb R)$, then ${\mathcal G}_v(k)$ is boundedly invertible in these spaces if and only if \cref{in.Dir} holds.
\end{lemma}

\begin{proof} It is known that ${\mathcal G}_0(k)$ is Fredholm of index zero in the aforementioned spaces, see \cite[Thm.7.17]{mclean:00}. Besides, it follows from \cref{mp.lem.S1S2est} that ${\mathcal G}_v(k)-{\mathcal G}_0(k)$ is compact in each of the aforementioned spaces, so that ${\mathcal G}_v(k)$ is Fredholm of index zero, since the index is invariant with respect to compact perturbations. Moreover, it follows from \cref{re.ImG+asm} that $\Im {\mathcal G}_0(k)$ has a smooth  kernel, which implies that $\Re {\mathcal G}_v(k)$ is also Fredholm of index zero. 

It follows from \cite[Thm.7.17]{mclean:00} and \cite[Thm.1.6]{nachman:88} that 
for any $-1 \leq s \leq 1$ the operator ${\mathcal G}_v(k)$ is invertible in $L\bigl( H^{s-\frac 1 2}(\partial\Omega),H^{s+\frac 1 2}(\partial\Omega) \bigr)$ if and only if $k^2$ is not a Dirichlet eigenvalue of $-\Delta+v(x)$.

Now, since $v$ satisfies \cref{in.v}, operator $-\Delta+v(x)$ has no embedded point spectrum in $L^2(\mathbb R^d)$ according to \cite[Thm.4.2]{agmon:75} and \cite[Thm.14.5.5 \& 14.7.2]{hoermander:83}. It follows that $R^+_v(k)$ has analytic continuation to a neighborhood of each $k>0$ in $\mathbb C$ and the this is also true for ${\mathcal G}_v(k)$ in view of formula \cref{mp.GviaR}. Hence, $\Re {\mathcal G}_v(k)$ is real analytic for $k>0$ and the same is true for $\Re {\mathcal G}_0(k)$.
\end{proof}

Let us introduce the operator
\begin{equation}\label{mp.T}
  W(k) := \begin{cases}\Re {\mathcal G}_0(k) - \mathcal E,& d \geq 3,\\
     \Re {\mathcal G}_0(k) - \mathcal E + \tfrac{1}{2\pi}\bigl( \ln \tfrac k 2 + \gamma\bigr)\langle 1, \cdot \rangle 1,& d = 2,
\end{cases}
\end{equation}
where $\gamma$ is the Euler-Mascheroni constant, and $\langle 1, \cdot \rangle$ denotes the scalar product with $1$ in $L^2(\partial\Omega)$.

\begin{lemma}\label{mp.lem.T} There exist $c_3 = c_3(\Omega)$, $k_3 = k_3(\Omega)$ such that
$W(k)$ belongs to $K\bigl( H^{-\frac 1 2}(\partial \Omega), H^{\frac 1 2}(\partial \Omega) \bigr)$ and 
\begin{gather*}
  \|W(k)\| \leq 
	 \begin{cases} k^2 c_3(\Omega) ,&         d =3 \mbox{ or } d\geq 5,\\
  k  |\ln k| c_3(\Omega),&    d = 4,\\
  k^2 |\ln k| c_3(\Omega), & d = 2
 \end{cases}
\end{gather*}
for all $k\in (0, k_3)$, 
where the norm is taken in $L\bigl( H^{-\frac 1 2}(\partial \Omega), H^{\frac 1 2}(\partial \Omega) \bigr)$.
\end{lemma}
\begin{proof} It follows from \cref{mp.ReG+asm} that $W(k) \in L\bigl(L^2(\partial \Omega),H^2(\partial \Omega) \bigr)$. By duality and approximation we get $W(k) \in K\bigl( H^{-\frac 1 2}(\partial \Omega), H^{\frac 1 2}(\partial \Omega)\bigr)$. The estimates also follow from \cref{mp.ReG+asm}, taking into account that $k\ln k = o(1)$ as $k \searrow 0$.
\end{proof}

\section{Derivation of the relations between $\Re{\mathcal G}_v$ and $\Im{\mathcal G}_v$}\label{sec:imreG}

\subsection{Proof of \cref{re.thm.ABr}}\label{re.thm.ABrp}
It follows from \cref{re.lem.Fred} for $v=0$ that if $\lambda >0$ is such that $\lambda^2$ is not a Dirichlet eigenvalue of $-\Delta$ in $\Omega$, then $Q(\lambda) \in L(H^1(\partial\Omega),L^2(\partial\Omega))$ is well defined as stated in \cref{re.Q}. It also follows from \cref{re.lem.Fred} that $\mathcal G_v(k) \in L(L^2(\partial\Omega),H^1(\partial\Omega))$ if \eqref{in.Dir} holds.

To prove the remaining assertions of \cref{re.thm.ABr} we suppose first that in addition to the initial assumptions of \cref{re.thm.ABr} the condition \cref{in.Dir} holds true for $v=0$. Let us define the Dirichlet-to-Neumann map $\Phi_v \in L\bigl(H^{1/2}(\partial\Omega),H^{-1/2}(\partial\Omega)\bigr)$ by
$\Phi_v f:=\frac{\partial \psi}{\partial \nu}$ where $\psi$ is the solution to 
\begin{align*}
&-\Delta \psi + v\psi = k^2\psi&&\mbox{in }\Omega,\\
&\psi = f&&\mbox{on }\partial\Omega,
\end{align*}
and $\nu$ is the outward normal vector on $\partial\Omega$. Moreover, let $\Phi_0$ denote the 
corresponding operator for $v=0$. It can be shown (see  e.g., \cite[Thm.1.6]{nachman:88}) 
that under the assumptions of \cref{re.thm.ABr} together with \eqref{in.Dir} for $v=0$ these operators are related 
to ${\mathcal G}$ and ${\mathcal G}_0$ as follows:
\begin{equation}\label{re.GPhi}
  {\mathcal G}_v^{-1} - {\mathcal G}_0^{-1} = \Phi_v - \Phi_0.
\end{equation}
For an operator $\mathcal T$ between complex function spaces let $\overline{\mathcal T}f:= \overline{\mathcal T\overline{f}}$ 
denote the operator with complex conjugate Schwarz kernel, and note that 
$\overline{\mathcal T}^{-1} = \overline{\mathcal T^{-1}}$ if $\mathcal T$ is invertible. Since $v$ is assumed to be real-valued, 
it follows that $\overline{\Phi}_v=\Phi_v$. Therefore, 
taking the complex conjugate in \cref{re.GPhi}, we obtain
\begin{equation*}
  (\overline {\mathcal G}_v)^{-1} - (\overline {\mathcal G}_0)^{-1} = \Phi_v - \Phi_0.
\end{equation*}
Combining the last two equations yields 
\begin{equation}\label{re.GG0}
  ({\mathcal G}_v^{-1}) - (\overline {\mathcal G}_v)^{-1} = ({\mathcal G}_0)^{-1} - (\overline {\mathcal G}_0)^{-1}.
\end{equation}
Together with the definitions \cref{re.ABQ} of $A,B$, and $Q$, we obtain the relation
\begin{equation*}
 (A+iB)iQ(A-iB) = -iB,
\end{equation*}
which can be rewritten as the two relations \cref{re.ABr}. Thus, \cref{re.thm.ABr} is proved under the additional assumption that \eqref{in.Dir} is satisfied for $v = 0$.

Moreover, it follows from formula \eqref{re.GG0} that the mapping \eqref{re.Q} extends to all $k>0$, i.e.\ the assumption that $k^2$ is not a Dirichlet eigenvalue of $-\Delta$ in $\Omega$ can be dropped. More precisely, for any $k>0$ one can always find $v$ satisfying \eqref{in.v}, \eqref{in.Dir} such that the expression on the hand side left of formula \eqref{re.GG0} is well-defined and can be used to define $Q(k)$. The existence of such $v$ follows from monotonicity and upper semicontinuity of Dirichlet eigenvalues.

This completes the proof of \cref{re.thm.ABr}.

\subsection{Proof of \cref{re.thm.AtBtr}}\label{re.thm.AtBtrp} 

\textit{Part I.} Let $k > 0$ be fixed. The fact that $T(k)$ extends continuously to $L^2(\partial \Omega)$ and is injective there is shown in \cite[Thm.3.28]{CK:13}. More precisely, the injectivity of $T(k)$ in 
$L(L^2(\partial\Omega),L^2(S^{d-1}))$ is proved in this book only in dimension $d=3$, but the proof works in any dimension $d\geq 2$. In addition, $T(k) \in L\bigl( L^2(\partial\Omega),L^2(S^{d-1}) \bigr)$ is compact as an operator with continuous integral kernel (see \cite[(3.58)]{CK:13}).

It follows from the considerations in the last paragraph of the proof of \cref{re.thm.ABr} that for any $k>0$ there exists $v \in C^\infty(\mathbb R^d,\mathbb R)$ with 
$\mathop{\mathrm{supp}} v \subset \Omega$ such that $k^2$ is not a Dirichlet eigenvalue of $-\Delta+v(x)$ in $\Omega$ and such that $Q(k) = \Im\mathcal G_v^{-1}(k)$.

Recall the formula
\begin{equation}\label{re.stone}
  \Im G^+_v(x,y,k) = c_1(d,k) \!\!\int\limits_{S^{d-1}} \!\!\psi^+_v(x,k\omega)\overline{\psi{}^+_v(y,k\omega)}\, ds(\omega),
	\quad c_1(d,k) := \tfrac{1}{8\pi} \left(\tfrac{k}{2\pi}\right)^{d-2}
\end{equation}
where $\psi^+_v(x,k\omega)$ is the total field corresponding to the incident plane wave $e^{ik\omega x}$ 
(i.e.\ $\psi^+_v(\cdot,k\omega)$ solves \eqref{in.eq} and $\psi^+_v(\cdot,k\omega)-e^{ik\omega\cdot}$ satisfies 
the Sommerfeld radiation condition \eqref{eq:SRC}), 
see, e.g., \cite[(2.26)]{melrose:95}. It follows that the operator $\Im \mathcal G_v(k)$ admits the factorization
\begin{equation*}
  \Im \mathcal G_v(k) = c_1(d,k) H_v(k) H_v^*(k),
\end{equation*}
where the operator $H_v(k) \in L( L^2(S^{d-1}), L^2(\partial\Omega))$ 
is defined as follows:
\begin{align*}
   \bigl( H_v(k) g \bigr) (x) & := \int_{S^{d-1}} \psi^+_v(x,k\omega) g(\omega) \, ds(\omega)
\end{align*}
Recall that $H_v(k)$ with $v=0$ is the Herglotz operator, see, e.g., \cite[(5.62)]{CK:13}.

\begin{lemma}\label{re.lem.C} Under the assumption \cref{in.Dir} $H^*_v(k)$ extends to a compact, 
injective operator with dense range in $L\bigl( H^{-1}(\partial\Omega),L^2(S^{d-1})\bigr)$. 
If $(Rh)(\omega) := h(-\omega)$, the following formula holds in $H^{-1}(\partial\Omega)$:
\begin{equation}\label{re.Cvfact}
  R \overline H_v^*(k) = \frac{1}{\sqrt{k}\,\cthree(d,k)} T(k) \mathcal G_v(k)
	\quad\mbox{with } \cthree(d,k):= \tfrac{1}{4\pi}\exp\left(-i\pi \tfrac{d-3}{4}\right)
	\left(\tfrac{k}{2\pi}\right)^{\frac{d-3}{2}}
\end{equation}
\end{lemma}
\begin{proof}[Proof of \cref{re.lem.C}]
We start from the following formula, which is sometimes referred to as mixed reciprocity relation (see \cite[(4.15)]{FM:85} or 
\cite[Thm.~3.16]{CK:13}):
\begin{equation*}
  G^+_v(x,y,k) = \cthree(d,k) \frac{e^{ik|x|}}{|x|^{(d-1)/2}} \psi^+_v\left(y,-k\tfrac{x}{|x|}\right) + 
	O\left(\frac{1}{|x|^{(d+1)/2}}\right), \quad |x|\to+\infty
\end{equation*}
This implies 
\begin{equation*}
 \begin{gathered}
 (G_v(k) h)(x) = \cthree(d,k)\frac{e^{i|k||x|}}{|x|^{(d-1)/2}} (R \overline H_v^*(k) h)(x) + O\left(\frac{1}{|x|^{(d+1)/2}} \right),\quad |x| \to + \infty 
 \end{gathered}
\end{equation*}
for $h \in L^2(\partial\Omega)$,
where $(G_v(k) \varphi)(x)$ is defined in the same way as ${\mathcal G}_v \varphi(x)$ in formula \cref{in.G0} 
but with $x \in \mathbb R^d \setminus{\overline \Omega}$, and from this we obtain \eqref{re.Cvfact} in 
$L^2(\partial\Omega)$. 

Recall also that $\mathcal G_v(k) \in GL(H^{-1}(\partial\Omega),L^2(\partial\Omega))$ if \cref{in.Dir} holds (see \cref{re.lem.Fred}). This together with injectivity of $T(k)$ in $L(L^2(\partial\Omega),L^2(S^{d-1}))$ and formula \cref{re.Cvfact} imply that $H^*_v(k)$ extends by continuity to a compact injective operator with dense range in $L\bigl( H^{-1}(\partial\Omega), L^2(S^{d-1}) \bigr)$ where it satisfies \cref{re.Cvfact}.
%
\end{proof}

Using \eqref{re.stone},  \eqref{re.Cvfact}, and the identities 
$R^*R=I=RR^*$ and $R=\overline{R}$, eq.~\eqref{eq:Q_factorization} can 
be shown as follows:
\begin{align*}
 -Q(k) &= \tfrac{-1}{2i}\left({\mathcal G}_v^{-1} - \overline {\mathcal G}_v^{-1}\right) 
   = {\mathcal G}_v^{-1} \Im \mathcal G_v \overline {\mathcal G}{}^{-1}_v 
  = c_1(d,k) \mathcal ( H_v^* \overline {\mathcal G}{}^{-1}_v )^* H_v^* \overline {\mathcal G}{}^{-1}_v \\
	& = \tfrac{c_1(d,k)}{k|\cthree(d,k)|^2}T^*(k) \overline{R}\overline{R}^* T(k) = T^*(k) T(k).
\end{align*}

\textit{Part II.} The operators $\widetilde A$, $\widetilde B \in L\bigl(L^2(\partial\Omega)\bigr)$ are compact in view of \cref{re.lem.Fred} and Part I of \cref{re.thm.AtBtr}. The relations \cref{re.AtBtr.1} and \cref{re.AtBtr.2} are direct consequences of \cref{re.ABr.1}, \cref{re.ABr.2} and of definition \cref{re.AtBt}.

\textit{Part III.} The operators $\widetilde A$, $\widetilde B \in L\bigl(L^2(\partial\Omega)\bigr)$ are real, compact symmetric and commute by \cref{re.AtBtr.2}. It is well known (see e.g.\ \cite[Prop. 8.1.5]{taylor:11} that under these conditions $\widetilde A$ and $\widetilde B$ must have a common eigenbasis in $L^2(\partial\Omega)$.

Moreover, if follows from \cref{re.AtBtr.1} that if $f \in L^2(\partial\Omega)$ is a common eigenfunction of 
$\widetilde A$ and $\widetilde B$, then the corresponding eigenvalues $\lambda_{\widetilde A}$ and 
$\lambda_{\widetilde B}$ of $\widetilde A$ and $\widetilde B$, respectively, satisfy the equation \cref{re.lAlBr}.

\Cref{re.thm.AtBtr} is proved.

\section{Stability of indices of inertia}\label{sec:inertia}
Let $S$ be a compact topological manifold (in what follows it will be $S^{d-1}$ or $\partial\Omega$). Let $A\in L\bigl(L^2(S)\bigr)$ be a real symmetric operator and suppose that
\begin{equation}\label{in.A}
 \text{$L^2_{\mathbb R}(S)$ admits an orthonormal basis of eigenvectors of $A$}.
\end{equation}
We denote this basis by $\{\varphi_n \colon n \geq 1\}$, i.e.\ $A\varphi_n=\lambda_n\varphi_n$. 
Property \cref{in.A} is obviously satisfied if $A$ is also compact. 
Let us define the projections onto the sum of the non negative and negative eigenspaces by 
\begin{equation}\label{eq:PA}
P^A_{+}x:= \sum\nolimits_{n\colon \lambda_n\geq 0}\langle x,\varphi_n\rangle \varphi_n\qquad 
P^A_{-}x:= \sum\nolimits_{n\colon \lambda_n<0}\langle x,\varphi_n\rangle \varphi_n 
\end{equation}
In addition, let $L^A_-$, $L^A_+$ denote the corresponding eigenspaces:
\begin{equation}\label{eq:LA}
  L^A_- = \mathop{\mathrm{ran}} P^A_-, \quad L^A_+ = \mathop{\mathrm{ran}} P^A_+.
\end{equation}
Then it follows from $Ax =\sum_{n=1}^\infty \lambda_n\langle x,\varphi_n\rangle \varphi_n$ that 
 \begin{gather*}
  \langle A x, x \rangle = |\langle A x_+, x_+ \rangle| - |\langle A x_-, x_-\rangle|\qquad \mbox{with }x_{\pm}:=P^A_{\pm}x.
 \end{gather*}
The numbers $\mathop{\mathrm{rk}} P^A_+$ and $\mathop{\mathrm{rk}} P^A_-$ in $\mathbb{N}_0\cup\{\infty\}$ (where $\mathbb N_0$ denotes the set of non-negative integers) are called 
\emph{positive and negative index of inertia} of $A$, and the triple $\mathop{\mathrm{rk}} P^A_+,\dim\mathop{\mathrm{ker}} A, \mathop{\mathrm{rk}} P^A_-$ 
is called \emph{inertia} of $A$. A generalization of the Sylvester inertia law to Hilbert spaces 
states that for a self-adjoint operator $A\in L(X)$ on a separable Hilbert space $X$ and an operator 
$\Lambda\in GL(X)$, the inertias of $A$ and $\Lambda^*A\Lambda$ coincide (see \cite[Thm.6.1, p.234]{cain:80}).  
We are only interested in the negative index of inertia, but we also have to consider operators $\Lambda$ 
which are not necessarily surjective, but only have dense range.

\begin{lemma}\label{in.lem.in}Let $S_1$, $S_2$ be two compact topological manifolds, $(A,\widetilde A,\Lambda)$ be a triple of operators
such that $A \in L(L^2(S_1))$ and $\widetilde A \in L(L^2(S_2))$ are real, symmetric, $\Lambda \in L(L^2(S_1),L^2(S_2))$,  $\widetilde A = \Lambda A \Lambda^*$ and 
$A$, $\widetilde A$ satisfy \cref{in.A}. Then $\mathop{\mathrm{rk}} P^{\widetilde A}_- \leq \mathop{\mathrm{rk}} P^A_-$. 
Moreover, if $\mathop{\mathrm{rk}} P^A_- < \infty$ and $\Lambda$ is injective, then $\mathop{\mathrm{rk}} P^{\widetilde A}_- = \mathop{\mathrm{rk}} P^A_-$.
\end{lemma}
\begin{proof} 
For each $x \in L^{\widetilde A}_- \setminus \{0\}$ we have
 \begin{align*}
  0 & > \langle \widetilde A x, x \rangle = \langle A \Lambda^* x, \Lambda^* x \rangle \\
   &= \bigl|\langle A P^A_+ \Lambda^* x, P^A_+ \Lambda^* x \rangle \bigr| - \bigl| \langle A P^A_- \Lambda^* x, P^A_- \Lambda^* x  \rangle \bigr| \\
  &\geq - \bigl| \langle A P^A_- \Lambda^* x, P^A_- \Lambda^* x \rangle \bigr|,
 \end{align*}
which shows that $P^A_- \Lambda^* x \neq 0$. Hence, the linear mapping $L^{\widetilde A}_- \to L^A_-$, 
$x \mapsto P^A_- \Lambda^* x$, is injective. This shows that $\mathop{\mathrm{rk}} P^{\widetilde A}_- \leq \mathop{\mathrm{rk}} P^A_-$.

Now suppose that $d = \mathop{\mathrm{rk}} P^A_- < \infty$ and that $\Lambda$ is injective. Note that the injectivity of $\Lambda$ implies that $\Lambda^*$ has dense range (see, e.g., \cite[Thm. 4.6]{CK:13}). Let $y_1$, \dots, $y_d$ be an orthonormal basis of $L^A_-$ with $A y_j = \lambda_j y_j$. Let $l_\text{min} = \min\{|\lambda_1|,\dots,|\lambda_d|\}$ and let $x_1$, \dots, $x_d$ be such that 
\begin{equation*}
 \| y_j - \Lambda^* x_j \|_2 < \varepsilon, \quad 5 \varepsilon \sqrt d \|A\| < l_\text{min}, \quad j=1,\dots,d.
\end{equation*}
Let $\alpha_1$, \dots, $\alpha_d \in \mathbb R$ be such that $\sum_j |\alpha_j|^2 = 1$ and put $x =\sum_j \alpha_j x_j$, $y = \sum_j \alpha_j y_j$. Note that
\begin{gather*}
 \| y - \Lambda^* x\|_2 \leq \varepsilon \sum\nolimits_j |\alpha_j| \leq \varepsilon \sqrt d, \\
 \langle Ay, y \rangle \geq \langle \widetilde Ax,x \rangle - 5 \varepsilon \sqrt d \|A\|  > \langle \widetilde Ax,x\rangle - l_\text{min}.
\end{gather*}   
Then we have
\begin{equation*}
 \begin{aligned}
  -l_\text{min} & \geq \langle A y, y \rangle > \langle \widetilde A x, x \rangle - l_\text{min}  \\
    & = \bigl| \langle \widetilde A P^{\widetilde A}_+ x, P^{\widetilde A}_+ x \rangle  \bigl| - \bigl| \langle \widetilde A P^{\widetilde A}_- x, P^{\widetilde A}_- x \rangle \bigr| - l_\text{min}  \\
    & \geq - \bigl| \langle \widetilde A P^{\widetilde A}_- x, P^{\widetilde A}_- x \rangle \bigr| - l_\text{min}, \quad j = 1,\dots,d.
 \end{aligned}
\end{equation*}
Hence, $P^{\widetilde A}_- x \neq 0$. Thus, the linear mapping $L^A_- \to L^{\widetilde A}_-$, defined on the basis by $y_j \to P^{\widetilde A}_- x_j$, $j = 1$, \dots, $d$, is injective and $\mathop{\mathrm{rk}} P^A_- \leq \mathop{\mathrm{rk}} P^{\widetilde A}_-$.
\end{proof}

The assumption \cref{in.A} in \cref{in.lem.in} can be dropped, but then the operators 
$P^{\widetilde A}_\pm$, $P^A_\pm$ must be defined using the general spectral theorem for self-adjoint operators.

The following two lemmas address the stability of the negative index of inertia under perturbations. 
We first look at perturbations of finite rank. 

\begin{lemma}\label{in.lem.add} Let $S$ be a compact topological manifold, let $A_1$, $A_2$ be compact self-adjoint operators in $L^2(S)$ 
and set $n_j:=\mathop{\mathrm{rk}} P^{A_j}_-$ for $j=1,2$. If $n_1 < \infty$ and $\mathop{\mathrm{rk}}(A_1 - A_2) < \infty$, 
then $n_2 \leq n_1 + \mathop{\mathrm{rk}}(A_1-A_2)$.
\end{lemma}
\begin{proof} Let $\lambda^{A_j}_1 \leq \lambda^{A_j}_2 \leq \cdots <0$ denote the negative eigenvalues of $A_j$ in 
$L^2(S)$, sorted in ascending order with multiplicities. By the min-max principle we have that
\begin{align*}
& \max_{S_{k-1}} \min_{x \in S_{k-1}^\bot, \|x\|=1} \langle A_j x, x\rangle = \lambda^{A_j}_k, \quad 1 \leq k \leq n_j, \\
  & \sup_{S_{k-1}} \min_{x \in S_{k-1}^\bot, \|x\|=1} \langle A_j x, x\rangle = 0, \quad k > n_j
\end{align*}
where the maximum is taken over all $(k-1)$-dimensional subspaces $S_{k-1}$ of $L^2(S)$ and $S^\bot$ denotes the orthogonal complement of $S_{k-1}$ in $L^2(S)$. Let $K = A_1-A_2$, $r = \mathop{\mathrm{rk}} K$ and note that $A_1x = A_2x$ for any $x \in \mathop{\mathrm{ker}} K = (\mathop{\mathrm{ran}} K)^\bot$. Also note that $(S_{k-1}\oplus \mathop{\mathrm{ker}} K)^\bot \subset \bigl(S_{k-1}^\bot \cap \overline{\mathop{\mathrm{ran}} K}\bigr)$. For $k = n_1 + 1$ we obtain 
\begin{align*}
  0 & = \sup_{S_{k-1}} \min_{x \in S_{k-1}^\bot, \|x\|=1} \langle A_1 x, x\rangle \\
    & \leq \sup_{S_{k-1}} \min \bigl\{ \langle A_2 x, x \rangle \colon x \in (S_{k-1}\oplus \mathop{\mathrm{ker}} K\bigr)^\bot, \|x\|=1  \bigr\} \\
    & \leq \sup_{S_{k-1+r}} \min_{x \in S_{k-1+r},\|x\|=1} \langle A_2 x,x \rangle 
		=\begin{cases}
		\lambda^{A_2}_{k+r} &\mbox{if }n_2\geq k+r,\\
		0 &\mbox{else}.
		\end{cases}
\end{align*}
Taking into account that $\lambda^{A_2}_{k+r}<0$, we obtain that only the second case is possible, and this implies $n_2 \leq k -1 +r$.
\end{proof}

In the next lemma we look at ``small'' perturbations. The analysis is complicated by the fact 
that we have to deal with operators with eigenvalues tending to $0$. 
Moreover, we not only have to show stability of $\mathop{\mathrm{rk}} P^A_-$ but also of 
$L^A_-$.

\begin{lemma}\label{in.lem.AB} Let $S_1$ be a $C^1$ compact manifold and $S_2$ a topological compact manifold. Let $(A,A_0,\Lambda)$ be a triple of operators such that $A, A_0 \in L\left(L^2(S_1)\right)$ are real, symmetric, satisfying \cref{in.A}; $\Lambda \in L\left(L^2(S_1),L^2(S_2)\right)$ is injective and
\begin{align*}
  & A_0 \in GL\left(H^{-\frac 1 2}(S_1),H^{\frac 1 2}(S_1)\right) 
	\cap GL\left(L^2(S_1),H^1(S_1)\right), \quad \mathop{\mathrm{rk}} P^{A_0}_- < \infty,\\
  &A - A_0 \in K\left(H^{-\frac 1 2}(S_1),H^{\frac 1 2}(S_1)\right)
\end{align*}
Put $\widetilde A = \Lambda A \Lambda^*$, $\widetilde A_0 = \Lambda A_0 \Lambda^*$. The following statements hold true:
\begin{enumerate}
 \item $\mathop{\mathrm{rk}} P^A_- < \infty$.
 \item For any $\sigma>0$ there exists $\delta = \delta(A_0,\Lambda,\sigma)$ such that if $\|A-A_0\| < \delta$ in 
$L\bigl(H^{-\frac 1 2}(S_1),H^{\frac 1 2}(S_1)\bigr)$, then
\begin{enumerate}
  \item $\mathop{\mathrm{rk}} P^{\widetilde A}_- = \mathop{\mathrm{rk}} P^{\widetilde A_0}_-$,
  \item $A$ is injective in $H^{-\frac 1 2}(S_1)$,
  \item if $\widetilde Af = {\widetilde \lambda} f$ for some $f \in L^2(S_2)$ with $\|f\|_2 = 1$, then $\lambda < 0$ if and only if $\mathop{\mathrm{d}}(f,L^{\widetilde A_0}_-) < \tfrac 1 2$,
  \item all negative eigenvalues of $\widetilde A$ in $L^2(S^{d-1})$ belong to the $\sigma$-neighborhood of negative eigenvalues of $\widetilde A_0$.
\end{enumerate}
\end{enumerate}
\end{lemma}

\begin{proof} \textit{First part.} We have that
\begin{equation*}
  A  = |A_0|^{\frac 1 2} \bigl( \mathrm{Id} + R + |A_0|^{-\frac 1 2}(A - A_0)|A_0|^{-\frac 1 2} \bigr)|A_0|^{\frac 1 2},
\end{equation*}
with $R$ finite rank compact operator in $L^2(S_1)$. More precisely, starting from the orthonormal eigendecomposition of $A_0$ in $L^2(S_1)$,
\begin{equation*}
  A_0 f = \sum\nolimits_{n=1}^\infty \lambda_n \langle f, \varphi_n \rangle \varphi_n,
\end{equation*}
we define $|A_0|^\alpha$ for $\alpha\in{\mathbb R}$ and $R$ as follows:
\begin{gather*}
  |A_0|^\alpha f = \sum\nolimits_{n=1}^\infty |\lambda_n|^\alpha \langle f, \varphi_n \rangle \varphi_n, \quad Rf = -2\sum\nolimits_{n:\lambda_n < 0} \langle f,\varphi_n\rangle \varphi_n.
\end{gather*}

By our assumptions and the polar decomposition, $|A_0|^{-1}$ is a symmetric operator on $L^2(S_1)$ with  domain $H^1(S_1)$, and $|A_0|^{-1}\in L(H^1(S_1),L^2(S_1))$. 
Consequently, by complex interpolation we get
\begin{equation*}
  |A_0|^{-\frac 1 2} \in L\left(H^{\frac 1 2}(S_1),L^2(S_1)\right).
\end{equation*}
In a similar way, we obtain
\begin{gather*}
  |A_0|^{-\frac 1 2} \in L\left(L^2(S_1),H^{-\frac 1 2}(S_1)\right), \qquad 
  |A_0|^{\frac 1 2} \in L\left( L^2(S_1), H^{\frac 1 2}(S_1) \right).
\end{gather*}
Thus, the operator $|A_0|^{-\frac 1 2}(A-A_0)|A_0|^{-\frac 1 2}$ is compact in $L^2(S_1)$. Hence, its eigenvalues converge to zero. Let us introduce the operators $D$, $D_0$ and $\Delta D$ by 
\begin{equation*}
  D := D_0 + \Delta D, \quad D_0 := \mathrm{Id} + R, \quad \Delta D := |A_0|^{-\frac 1 2}(A - A_0)|A_0|^{-\frac 1 2}, 
\end{equation*}
Then the eigenvalues of $D$
converge to $1$, and only finite number of eigenvalues of $D$ in $L^2(S_1)$ can be negative. Applying \cref{in.lem.in} to the triple $(D,A,|A_0|^{\frac 1 2})$, we get the first statement of the present lemma.

\textit{Second part.}
At first, we show that there exists $\delta'$ such that if $\|A-A_0\| < \delta'$, then $\mathop{\mathrm{rk}} P^A_- \leq \mathop{\mathrm{rk}} P^{A_0}_-$. Here the norm is taken in $L\bigl( H^{-\frac 1 2}(S_1), H^{\frac 1 2}(S_1) \bigr)$. Note that the spectrum of $D_0$ in $L^2(S_1)$ consists at most of the two points $-1$ and $1$. 
Thus, the spectrum $\sigma^D$ of $D$ satisfies
\begin{equation*}
 \sigma^D \subseteq [-1-\|\Delta D\|,-1+\|\Delta D\|] \cup [1-\|\Delta D\|,1+\|\Delta D\|],
\end{equation*}
where $\|\Delta D\|$ is the norm of $\Delta D$ in $L\bigl(L^2(S_1)\bigr)$. It follows that if $x \in L^D_-$, $\|x\|_2 = 1$, then
\begin{equation*}
  \langle Dx,x \rangle \leq -1 + \|\Delta D\|. 
\end{equation*}
On the other hand,
\begin{equation*}
\begin{gathered}
  \langle Dx,x\rangle \geq \langle D_0x,x\rangle - \|\Delta D\| 
  \geq -\bigl|\langle D_0 x_-, x_- \rangle \bigr| - \|\Delta D\|
\end{gathered}
\end{equation*}
for $x_- = P^{D_0}_- x$. It follows from the last two inequalities that
\begin{equation*}
  \bigl| \langle D_0 x_-, x_-\rangle \bigr| \geq 1 - 2 \|\Delta D\|.
\end{equation*}
Thus, if $\|\Delta D\| < \frac 1 2$, the mapping $L^D_- \to L^{D_0}_-$, $x \to P^{D_0}_-x$ is injective, so $\mathop{\mathrm{rk}} P^D_- \leq \mathop{\mathrm{rk}} P^{D_0}_-$. Using \cref{in.lem.in} to the triple $(D,A,|A_0|^{\frac 1 2})$ and taking into account that $\mathop{\mathrm{rk}} P^{D_0}_- = \mathop{\mathrm{rk}} P^{A_0}_-$ we also get that $\mathop{\mathrm{rk}} P^A_- \leq \mathop{\mathrm{rk}} P^{A_0}_-$. 
Moreover, there exists $\delta' = \delta'(A_0,\Lambda)$ such that if $\|A-A_0\| < \delta'$ in the norm of 
$L\bigl(H^{-\frac 1 2}(S_1),H^{\frac 1 2}(S_1)\bigr)$, then $\|\Delta D\| < \tfrac 1 2$ 
and consequently, $\mathop{\mathrm{rk}} P^A_- \leq \mathop{\mathrm{rk}} P^{A_0}_-$. In addition, taking into account that $|A_0|^{-\frac 1 2} \in GL\bigl(L^2(S_1),H^{-\frac 1 2}(S_1)\bigr)$, we obtain that $A$ is injective in $H^{-\frac 1 2}(S_1)$ if $\|A-A_0\|<\delta'$.

Applying \cref{in.lem.in} to the triple $(A,\widetilde A,\Lambda)$ and using the assumption that $\Lambda$ is injective, we obtain that $\mathop{\mathrm{rk}} P^{\widetilde A}_- \leq \mathop{\mathrm{rk}} P^{\widetilde A_0}_-$ if $\|A-A_0\|<\delta'$ in $L\bigl(H^{-\frac 1 2}(S_1),H^{\frac 1 2}(S_1)\bigr)$.

Now let $\Sigma$ be the union of circles of radius $\sigma>0$ in $\mathbb C$ centered at negative eigenvalues of $\widetilde A_0$ in $L^2(S_2)$. It follows from \cite[Thm.3.16 p.212]{kato:80} that there exists $\delta'' = \delta''(A_0,\Lambda,\sigma)$, $\delta'' < \delta'$, such that if $\|A-A_0\| < \delta''$, then $\Sigma$ also encloses $\mathop{\mathrm{rk}} P^{\widetilde A_0}_-$ negative eigenvalues of $\widetilde A$.

Taking into account that $\mathop{\mathrm{rk}} P^{\widetilde A}_- \leq \mathop{\mathrm{rk}} P^{\widetilde A_0}_-$ if $\|A-A_0\| < \delta''$, we get that $\mathop{\mathrm{rk}} P^{\widetilde A}_- = \mathop{\mathrm{rk}} P^{\widetilde A_0}_-$. In addition, it follows from \cite[Thm.3.16 p.212]{kato:80} that there exists $\delta'''=\delta'''(A_0,\Lambda,\sigma)$, $\delta''' < \delta''$, such that if $\|A-A_0\|<\delta'''$, then $\|P^{\widetilde A}_- - P^{\widetilde A_0}_-\| < \tfrac 1 2$.

The second statement now follows from the following standard fact:
\begin{lemma} The following inequalities are valid:
\begin{align*}
&\mathop{\mathrm{d}}(f,L^{\widetilde A_0}_-) \leq \|P^{\widetilde A}_--P^{\widetilde A_0}_-\|
&&\mbox{for all }f \in L^{\widetilde A}_-\mbox{ with }\|f\|_2=1\quad \mbox{and}\\
&\mathop{\mathrm{d}}(f,L^{\widetilde A_0}_-) \geq 1 - \|P^{\widetilde A}_- - P^{\widetilde A_0}_-\|
&&\mbox{for all }f \in L^{\widetilde A}_+ \mbox{ with }\|f\|_2=1.
\end{align*}
\end{lemma}

\Cref{in.lem.AB} is proved.
\end{proof}

\section{Derivation of the uniqueness results}\label{sec:uniqueness}
The proof of the uniqueness theorems will be based on the following two propositions: 
\begin{proposition}\label{un.prp.hel} 
For all $\kappa \in L^\infty(\Omega,\mathbb R)$ and all $\omega>0$ the operator
$\Re {\mathcal G}_{-\omega^2\kappa}(\omega)$ (resp. $\Re \widetilde {\mathcal G}_{-\omega^2 \kappa}(\omega)$) can have at most a finite number of negative eigenvalues in $L^2(\partial \Omega)$ (resp. $L^2(S^{d-1})$), multiplicities taken into account. In addition, for all $M>0$ there exists a constant $\omega_0 = \omega_0(\Omega,M)$ such that for all $\kappa$ satisfying $\|\kappa\|_\infty \leq M$ the condition
\cref{in.Dir} with $v = -\omega^2 \kappa$ is satisfied and the operator $\Re {\mathcal G}_{-\omega^2 \kappa}(\omega)$ (resp. $\Re \widetilde {\mathcal G}_{-\omega^2 \kappa}(\omega)$) is positive definite on $L^2(\partial\Omega)$ (resp. $L^2(S^{d-1})$) if $\omega \in (0,\omega_0]$.
\end{proposition}

\begin{proof}
\textit{Step 1.} We are going to prove that $\Re {\mathcal G}_{-\omega^2 \kappa}(\omega)$ can have only finite number of negative eigenvalues in $L^2(\partial\Omega)$ and, in addition, there exists $\omega_0' = \omega_0'(\Omega,M)$ such that if $\omega \in (0,\omega_0']$, then $\Re {\mathcal G}_{-\omega^2 \kappa}(\omega)$ is positive definite in $L^2(\partial\Omega)$.

Let $\mathcal E$ be defined according to \cref{mp.E}. The operator $\mathcal E$ is positive definite in $L^2(\partial D)$ for $d \geq 3$, see, e.g., \cite[Cor.8.13]{mclean:00}. For the case $d=2$, the operator 
\begin{equation*}
\mathcal E_r = \mathcal E + \tfrac{1}{2\pi} \langle 1,\cdot\rangle \ln r
\end{equation*}
is positive definite in $L^2(\partial \Omega)$  if and only if $r > \mathrm{Cap}_{\partial \Omega}$, where $\mathrm{Cap}_{\partial \Omega}$ denotes the capacity of $\partial \Omega$, see, e.g., \cite[Thm.8.16]{mclean:00}. We consider the cases $d \geq 3$  and $d = 2$ separately.

\textit{$d \geq 3$.} We have that
\begin{equation*}
  \mathcal E \in GL\bigl( H^{-\frac 1 2}(\partial \Omega), H^{\frac 1 2}(\partial \Omega) \bigr) \cap GL\bigl( L^2(\partial \Omega), H^1(\partial \Omega) \bigr).
\end{equation*}
This follows from \cite[Thm.7.17 \& Cor.8.13]{mclean:00}. Using \cref{mp.lem.S1S2est} and \cref{mp.lem.T} we also get that
\begin{align*}
&\mathcal E - \Re {\mathcal G}_{-\omega^2\kappa}(\omega) \in K\bigl( H^{-\frac 1 2}(\partial \Omega), H^{\frac 1 2}(\partial \Omega) \bigr)\mbox{ with} \\
&\| \mathcal E - \Re {\mathcal G}_{-\omega^2\kappa}(\omega) \| \leq \omega^2 c_2'(\Omega,M)\|\kappa\|_\infty + \omega |\ln \omega| c_3'(\Omega)
\end{align*}
for all $\omega\in(0, \min\{ k_2(\Omega,M), k_3(\Omega) \})$, 
with the norm in $L\bigl(H^{-\frac 1 2}(\partial \Omega),H^{\frac 1 2}(\partial \Omega) \bigr)$.

Applying \cref{in.lem.AB} to the triple $\bigl(\Re {\mathcal G}_{-\omega^2\kappa}(\omega),\mathcal E,\mathrm{Id}\bigr)$, we find that $\Re {\mathcal G}_{-\omega^2 \kappa}(\omega)$ can have at most finite number of negative eigenvalues in $L^2(\partial\Omega)$ and that there exists $\omega_0' = \omega_0'(\Omega,M)$ such that $\Re {\mathcal G}_{-\omega^2\kappa}(\omega)$ is 
positive definite in $L^2(\partial \Omega)$ if $\omega \in (0,\omega_0']$.

\textit{$d = 2$.} Let $r > \mathrm{Cap}_{\partial \Omega}$. We have that
\begin{equation*}
  \mathcal E_r \in GL\bigl( H^{-\frac 1 2}(\partial \Omega), H^{\frac 1 2}(\partial \Omega) \bigr) \cap GL\bigl( L^2(\partial \Omega), H^1(\partial \Omega) \bigr).
\end{equation*}
This follows from \cite[Thm.7.18 \& Thm.8.16]{mclean:00}. Using \cref{mp.lem.S1S2est} and \cref{mp.lem.T} we also have that
\begin{gather*}
  \mathcal E_r - \Re {\mathcal G}_{-\omega^2\kappa}(\omega) \in K\bigl( H^{-\frac 1 2}(\partial \Omega), H^{\frac 1 2}(\partial \Omega) \bigr). 
\end{gather*}
Note that 
\begin{equation}\label{un.SEr}
\begin{gathered}
  \Re {\mathcal G}_{-\omega^2\kappa}(\omega) - \mathcal E_r = \bigl( \Re {\mathcal G}_{-\omega^2\kappa}(\omega) - \Re {\mathcal G}_0(\omega) \bigr) 
  + W(\omega) - \tfrac{1}{2\pi} \bigl( \ln \tfrac{\omega r}{2} + \gamma \bigr) \langle 1, \cdot \rangle 1,
\end{gathered}
\end{equation}
where $W(\omega)$, $\gamma$ are defined according to \cref{mp.T}. Fix $r > \mathrm{Cap}_{\partial \Omega}$. Using \cref{mp.lem.S1S2est}, \cref{mp.lem.T} and formula \cref{un.SEr} we obtain that
\begin{equation*}
 \begin{aligned}
  &\| R(\omega) - \mathcal E_r \| \leq \omega^2 |\ln \omega|^2 c_2'(\Omega,M)\|\kappa\|_\infty + \omega^2 |\ln \omega| c_3'(\Omega), \\
 &\mbox{with } R(\omega) := \Re {\mathcal G}_{-\omega^2\kappa}(\omega) + \tfrac{1}{2\pi} \bigl( \ln \tfrac{\omega r}{2} + \gamma \bigr)\langle 1, \cdot \rangle 1 
 \end{aligned}
\end{equation*}
for all $\omega \in (0, \min\{ k_2(\Omega,M),k_3(\Omega)\})$
with the norm in $L\bigl(H^{-\frac 1 2}(\partial \Omega),H^{\frac 1 2}(\partial \Omega) \bigr)$.

Applying \cref{in.lem.AB} to the triple $\bigl(R(\omega), \mathcal E_r, \mathrm{Id}\bigr)$, we find that $R(\omega)$ can have only finite number of negative eigenvalues in $L^2(\partial \Omega)$ and, in addition, there exists $\omega_0' = \omega_0'(\Omega,M,r)$ such that if $\omega \in (0,\omega_0']$, then $R(\omega)$ is positive definite in $L^2(\partial \Omega)$. 

Applying \cref{in.lem.add} to the pair of operators $\bigl(R(\omega),\Re {\mathcal G}_{-\omega^2\kappa}(\omega)\bigr)$ we obtain that $\Re {\mathcal G}_{-\omega^2\kappa}(\omega)$ can have only finite number of negative eigenvalues in $L^2(\partial\Omega)$, since it is true for $R(\omega)$ and $\tfrac{1}{2\pi} \bigl( \ln \tfrac{\omega r}{2} + \gamma \bigr) \langle 1, \cdot \rangle 1$ is a rank one operator. Assuming, without loss of generality, that $\omega_0' < \tfrac 2 r e^{-\gamma}$, one can also see that the operator $\Re {\mathcal G}_{-\omega^2\kappa}(\omega)$ is positive definite for $\omega \in (0,\omega_0']$, as long as $R(\omega)$ is positive definite and $-\tfrac{1}{2\pi} \bigl( \ln \tfrac{\omega r}{2} + \gamma \bigr)\langle 1, \cdot \rangle 1$ is non-negative definite.

\textit{Step 2.} Applying \cref{in.lem.in} to the triple 
\begin{equation*}
 (\Re{\mathcal G}_{-\omega^2\kappa}(\omega),\Re \widetilde {\mathcal G}_{-\omega^2\kappa}(\omega),T(\omega)),
\end{equation*}
we find that the operator $\Re \widetilde {\mathcal G}_{-\omega^2\kappa}(\omega)$ can have only finite number of negative eigenvalues in $L^2(S^{d-1})$ and, in addition, there exists $\omega_0 = \omega_0(\Omega,M)$, $\omega_0 < \omega_0'$, such that $\Re \widetilde{\mathcal G}_{-\omega^2\kappa}(\omega)$ is positive definite in $L^2(S^{d-1})$ if $\omega \in (0,\omega_0]$.
\end{proof}

\begin{proposition}\label{un.prp.sch}  Let $v$, $v_0 \in L^\infty(\Omega,\mathbb R)$. 
Suppose that $k>0$ is such that $\Re {\mathcal G}_{v_0}(k)$ is injective in $H^{-\frac 1 2}(\partial \Omega)$ and \cref{un.Geig} holds true for $v_0$. Moreover, let $L^{v_0}_-$ denote the linear space spanned by the eigenfunctions of $\Re \widetilde {\mathcal G}_{v_0}(k)$ corresponding to negative eigenvalues, and
let $f \in L^2(S^{d-1})$ with $\|f\|_2=1$, be such that $\Re \widetilde {\mathcal G}_v\bigr(k) f = \lambda f$.
Then for any $\sigma>0$ there exists $\delta = \delta(\Omega,k,v_0,\sigma)$ such that if $\|v - v_0\|_\infty \leq \delta$, then
\begin{enumerate}
 \item $\Re \mathcal G_v(k)$ is injective in $H^{-\frac 1 2}(\partial\Omega)$,
 \item \cref{un.Geig} holds true for $v$,
 \item $\lambda < 0$ if and only if $\mathop{\mathrm{d}}(f,L^{v_0}_-) < \tfrac 1 2$,
 \item all negative eigenvalues of $\Re \widetilde{\mathcal G}_v(k)$ in $L^2(S^{d-1})$ belong to the $\sigma$-neighborhood of negative eigenvalues of $\Re \widetilde{\mathcal G}_{v_0}(k)$.
\end{enumerate}
\end{proposition}

\begin{proof}
Put
\begin{equation*}
 A := \Re {\mathcal G}_v(k), \; A_0 := \Re {\mathcal G}_{v_0}(k), \; 
\widetilde A:= \Re \widetilde {\mathcal G}_v(k), \; \widetilde 
A_0 := \Re \widetilde {\mathcal G}_{v_0}(k).
\end{equation*}
It follows from \cref{un.prp.hel} that $\mathop{\mathrm{rk}} P^{A_0}_- < \infty$.

Using \cref{re.lem.Fred} and the injectivity of $A_0$ in $H^{-1/2}(\partial \Omega)$, we obtain that
\begin{equation*}
 A_0 \in GL\bigl( H^{-1/2}(\partial \Omega),H^{1/2}(\partial \Omega) \bigr) \cap GL\bigl( L^2(\partial \Omega), H^1(\partial \Omega)\bigr).
\end{equation*}
It also follows from \cref{mp.lem.S1S2est} that
\begin{equation*}
  A - A_0 \in K\bigl(H^{-1/2}(\partial \Omega), H^{1/2}(\partial \Omega) \bigr).
\end{equation*}
Applying \cref{in.lem.AB} to the triple $\bigl(A,A_0,T\bigr)$, we find that there exists $\delta' = \delta'(\Omega,k,v_0)$ such that if $\|A-A_0\| \leq \delta'$ in $L\bigl( H^{-1/2}(\partial \Omega), H^{1/2}(\partial\Omega) \bigr)$, then $\mathop{\mathrm{rk}} P^{\widetilde A}_- = \mathop{\mathrm{rk}} P^{\widetilde A_0}_-$ and $A$ is injective in $H^{-1/2}(\partial\Omega)$.
Moreover, if $\widetilde A f = \lambda f$ with 
$\|f\|_2 = 1$, then $\lambda<0$ if and only if $\mathop{\mathrm{d}}(f,L^{v_0}_-) < \frac 1 2$.  Also note that, in view of \cref{mp.lem.S1S2est}, there exists 
$\delta = \delta(\Omega,k,v_0)$ such that if $\|v-v_0\|_\infty\leq \delta$, 
then $\|A-A_0 \| \leq \delta'$.

It remains to show that if $\|v-v_0\|_\infty \leq \delta$ for $\delta$ small enough and \cref{un.Geig} holds true for $v_0$, then it also holds true for $v_0$. But this property follows from the upper semi-continuity of a finite number of eigenvalues of $\widetilde{\mathcal G_v}(k)$ with respect to perturbations (see \cite[Thm.3.16 p.212]{kato:80}), from \cref{mp.lem.S1S2est} and from the fact that $\widetilde{\mathcal G_v}(k)$ has at most a finite number of eigenvalues with negative real part (see \cref{un.prp.hel} with $-\omega^2\kappa = v$, $\omega=k$).

\Cref{un.prp.sch} is proved.

\end{proof}

\subsection{Proof of \cref{un.thm.sch}}\label{un.thm.schp}

Let $k>0$ and $v_0$ be the same as in the formulation of \cref{un.thm.sch}. 
It follows from 
\cref{un.prp.hel} with $v_0 = -\omega^2 \kappa$ that the operator $\Re \widetilde {\mathcal G}_{v_0}(k)$ can have only finite number of negative eigenvalues in $L^2(\partial \Omega)$, multiplicities taken into account. 

Let $\delta = \delta(\Omega,k,v_0)$ be choosen as in \cref{un.prp.sch}. Suppose that $v_1$, $v_2$ are two functions satisfying the conditions of \cref{un.thm.sch} and put 
\begin{equation*}
\widetilde A_j := \Re \widetilde {\mathcal G}_{v_j}(k), \quad
\widetilde B_j := \Im \widetilde {\mathcal G}_{v_j}(k), \quad j = 1, 2. 
\end{equation*}
By the assumptions of the present theorem, 
\begin{equation*}
  \widetilde B_1 = \widetilde B_2.
\end{equation*}
Together with \cref{re.thm.AtBtr} and formula \cref{un.Geig} it follows that the operators $\widetilde A_1$ and $\widetilde A_2$ have a common basis of eigenfunctions in $L^2(S^{d-1})$ and that if $\widetilde A_1 f = \lambda_1 f$, $\widetilde A_2 f = \lambda_2 f$, for some $f \in L^2(S^{d-1})$, $\|f\|_2 = 1$, then
\begin{equation}\label{re.lA1lA2}
  |\lambda_1| = |\lambda_2|.
\end{equation}
More precisely, any eigenbasis of $\widetilde B_1$ is a common eigenbasis for $\widetilde A_1$ and $\widetilde A_2$.

It follows from \cref{un.prp.sch} that $\lambda_1 < 0$ if and only if $\mathop{\mathrm{d}}(f,L^{v_0}_-) < \tfrac 1 2$, and the same condition holds true for $\lambda_2$. Hence, $\lambda_1 < 0$ if and only if $\lambda_2 < 0$. Thus, we have 
\begin{equation}\label{re.At1At2}
  \widetilde A_1 = \widetilde A_2.
\end{equation}
Since by \cref{re.thm.AtBtr} (I) the operator $T$ is injective with dense range the same is true for $T^*$ by \cite[Thm. 4.6]{CK:13}. Injectivity of $T$ and \cref{re.At1At2} imply $A_1 T^* = A_2 T^*$. This equality, density of the range of $T^*$ and continuity of $A_1$, $A_2$ now imply that $A_1 = A_2$ and hence ${\mathcal G}_{v_1}(k) = {\mathcal G}_{v_2}(k)$.

Now we can use that fact that ${\mathcal G}_{v_1}(k)={\mathcal G}_{v_2}(k)$ implies $v_1=v_2$ if \eqref{in.Dir} holds true for $v = v_1$ and $v=v_2$, see \cite{novikov:88,berezanskii:58}. In turn, property \eqref{in.Dir} for $v = v_j$ follows from injectivity of $\Re \mathcal G_{v_j}(k)$ in $H^{-1/2}(\partial\Omega)$ (see \cref{un.prp.sch}). This completes the proof of \cref{un.thm.sch}.

\subsection{Proof of \cref{un.thm.hel}}

Let $\omega_0 = \omega_0(\Omega,M)$ be as in \cref{un.prp.hel} and let $\omega \in (0,\omega_0]$ be fixed. Put 
\begin{equation*}
  \widetilde A_j := \Re \widetilde {\mathcal G}_{-\omega^2 \kappa_j}(\omega), \quad
  \widetilde B_j := \Im \widetilde {\mathcal G}_{-\omega^2 \kappa_j}(\omega), \quad j = 1, 2.
\end{equation*}

It follows from \cref{un.prp.hel} that all the eigenvalues of $\widetilde{\mathcal G}_{-\omega^2\kappa}(\omega)$ have positive real parts so that condition \cref{un.Geig} is valid. As in the proof of \cref{un.thm.sch} one can show that $\widetilde A_1$, $\widetilde A_2$ have a common basis of eigenfunctions in $L^2(S^{d-1})$ (any eigenbasis of $\widetilde B_1$ is a common eigenbasis for $\widetilde A_1$ and $\widetilde A_2$) and the relation \cref{re.lA1lA2} holds.

In view of \cref{un.prp.hel} we also have that $\lambda_1 > 0$, $\lambda_2 >0$ such that $\lambda_1 = \lambda_2$. Thus, \cref{re.At1At2} holds true. 
Starting from equality \cref{re.At1At2} and reasoning as in the end of the proof of \cref{un.thm.sch}, we obtain that $\kappa_1 = \kappa_2$, completing the proof of \cref{un.thm.hel}.

\section{Discussion of the assumptions of Theorem \ref{un.thm.sch}}\label{un.lem.genp}
The aim of this section is to present results indicating that the assumptions of Theorem \ref{un.thm.sch} are always satisfied except 
for a discrete set of exceptional parameters. As a first step we characterize the adjoint operator 
$\widetilde{\mathcal G}_v^*(k)$ as a farfield operator for the scattering of distorted plane waves at 
$\Omega$ with Dirichlet boundary conditions. 
Note that, in particular, $-\frac{1}{k\overline{\cthree(d,k)}}\widetilde{\mathcal G}_0^*(k)$ 
is a standard farfield operator for Dirichlet scattering at $\Omega$ (see e.g.\ \cite[\S 3.3]{CK:13}).

\begin{lemma}\label{ge.lem.sct} Let $v$ satisfy \eqref{in.v} and consider $\psi^+_v$ and $\cthree$ as 
defined as in \cref{re.thm.AtBtrp}. 
Then we have $k\overline{\cthree(d,k)}\widetilde{\mathcal G}_v^*(k)g = u_\infty$  
for any $g \in L^2(S^{d-1})$ 
where $u_\infty \in L^2(S^{d-1})$ is the farfield pattern of the solution $u$ to the exterior boundary 
value problem \eqref{eqs:extBVP} with boundary values
\[
u_0(x) = \int_{S^{d-1}} \overline{\psi^+_v(x,-k\omega)} g(\omega) \, ds(\omega), \quad x \in \partial\Omega.
\]
\end{lemma}
\begin{proof} It follows from the definition of operators  $H_v$, $T$, $R$ in \cref{re.thm.AtBtrp} that 
$u_0 = \overline{H_v}Rg$ and $u_\infty = k^{-1/2}T \overline{H_v} R g$.
Using eq.~\cref{re.Cvfact} in Lemma \ref{re.lem.C} we also find that 
$\sqrt{k}\,\overline{\cthree(d,k)}\overline{H_v} R = \mathcal G_v^* T^*$. 
Hence, 
$k\,\overline{\cthree(d,k)}u_\infty = T\mathcal G_v^*T^*g
= \widetilde{\mathcal G}_v^* g$.
\end{proof}

\begin{lemma}\label{un.lem.gen} Let $\Omega$ satisfy \eqref{in.dom} and suppose that $\Omega$ is stricty starlike in the sense that $x \nu_x > 0$ for all $x \in \partial\Omega$, where $\nu_x$ is the unit exterior normal to $\partial\Omega$ at $x$. Let $v \in L^\infty(\Omega,\mathbb R)$ and let $k>0$ be such that $k^2$ is not a Dirichlet eigenvalue of $-\Delta$ in $\Omega$. Then there exist $M = M(k,\Omega)>0$, $\varepsilon=\varepsilon(k,\Omega)>0$, such that if $\|v\|_\infty \leq M$, then $\widetilde{\mathcal G}_v(\xi)$ satisfies \eqref{un.Geig} for all but a finite number of $\xi \in [k,k+\varepsilon)$.	
\end{lemma}

\begin{proof}
\textit{Part I.} We first consider the case $v=0$. It follows from \cref{ge.lem.sct} together with the equality $\psi^+_0(x,k\omega) = e^{ik\omega x}$ that the operator $\widetilde{\mathcal G}_0^*(k)$ is the farfield operator for the classical obstacle scattering problem with obstacle $\Omega$. Moreover,
\begin{equation*}
  S(k) := \mathrm{Id} - 2i \overline{\widetilde{\mathcal G}^*_0(k)},
\end{equation*}  
is the scattering matrix in the sense of \cite{HR:76}. It follows from \cite[(2.1) and the remark after (1.9)]{HR:76} that all the eigenvalues $\lambda \neq 1$ of $S(k)$ move in the counter-clockwise direction on the circle $|z|=1$ in $\mathbb C$ continuously and with strictly positive velocities as $k$ grows. More precisely, if $\lambda(k) = e^{i\beta(k)}$, $\lambda(k) \neq 1$ is an eigenvalue of $S(k)$ corresponding to the normalized eigenfunction $g(\cdot,k)$, then
\begin{gather*}
	\frac{\partial \beta}{\partial k}(k) = \frac{1}{4\pi}\left(\frac{k}{2\pi}\right)^ {d-2}\int_{\partial\Omega} \left| \frac{\partial f}{\partial \nu_x}(x,k) \right|^2 x \nu_x \, ds(x), \\ 
	f(x,k) = \int_{S^{d-1}} g(\theta,k) \bigl( e^{-ik\theta x} - u(x,\theta) \bigr)\, ds(\theta), \quad x \in \mathbb R^d \setminus \Omega,
\end{gather*}
where $u(x,\theta)$ is the solution of problem \cref{eqs:extBVP} with $u_0(x) = e^{-ik\theta x}$
(note that \cite{HR:76} uses a different sign convention in the radiation condition \eqref{eq:SRC} resulting 
in a different sign of $\partial\beta/\partial k$). 
It follows from this formula that $\frac{\partial \beta}{\partial k}(k)>0$:
\begin{enumerate}
	\item the term $x \eta_x $ is positive by assumption,
	\item $\frac{\partial f}{\partial \nu_x}$ cannot vanish on $\partial\Omega$ identically. Otherwise, $f$ vanishes on the boundary together with its normal derivative, and Huygens' principle (see \cite[Thm.3.14]{CK:13}) 
	implies that the scattered field 
$
	\int_{S^{d-1}} g(\theta,k) u(x,\theta) \, ds(\theta)
$
vanishes identically, so that $f$ is equal to
$
	f(x,k) = \int_{S^{d-1}} g(\theta,k) e^{-ik\theta x} \, ds(\theta).
$
One can see from this formula that $f$ extends uniquely to an entire solution of $-\Delta f = k^2 f$. Moreover, $f$ is a Dirichlet eigenfunction for $\Omega$ and it implies that $f$ is identically zero, because $k^2$ is not a Dirichlet eigenvalue of $-\Delta$ in $\Omega$ by assumption. Now it follows \cite[Thm.3.19]{CK:13} that the Herglotz kernel $g(\cdot,k)$ of $f$ vanishes, but it contradicts the fact that $g(\cdot,k)$ is a normalized eigenfunction of $S(k)$.
\end{enumerate}
It follows that all the non-zero eigenvalues of $\widetilde{\mathcal G}_0^*(k)$ move continuously in the clockwise direction on the circle $|z+i/2| = 1/2$ in $\mathbb C$ with non-zero velocities as $k$ grows. Moreover, since $\widetilde{\mathcal G}^*_0(k)$ is compact in $L^2(S^{d-1})$ (see \cref{re.thm.AtBtr}), it follows that $z=0$ is the only accumulation point for eigenvalues of $\widetilde{\mathcal G}^*_0(k)$. This together with \cref{un.prp.hel} for $\kappa = 0$, $\omega = k$, implies that there exist $\delta(k,\Omega) > 0$, $\varepsilon(k,\Omega)>0$ such that all the eigenvalues $\lambda$ of $\widetilde{\mathcal G}^*_0(\xi)$ with $\Re \lambda < 0$ belong to the half plane $\Im z < -\delta$ for $\xi \in [k,k+\varepsilon)$. 

This proves \cref{un.lem.gen} with $v = 0$ if we take into account that the eigenvalues of $\widetilde{\mathcal G}_0^*(k)$ and $\widetilde{\mathcal G}_0(k)$ are related by complex conjugation.

\textit{Part II.} Let $k$ be such that \cref{un.Geig} holds true for $v=0$ and choose $\delta(k,\Omega)$, $\varepsilon(k,\Omega)$ as in the first part of the proof. Now let $v \in L^\infty(\Omega,\mathbb R)$. It follows from \cref{un.prp.sch} that for any $\sigma > 0$ there exists $M=M(\xi,\sigma)$ such that if $\|v\|_\infty \leq M$, then $\widetilde{\mathcal G}^*_v(\xi)$ has a finite number of eigenvalues $\lambda$ with $\Re \lambda < 0$, multiplicities taken into account, and these eigenvalues belong to the $\sigma$-neighborhood of the eigenvalues of $\widetilde{\mathcal G}_0^*(\xi)$ if $\xi \in [k,k+\varepsilon)$. In addition, $M(\xi,\sigma)$ can be choosen depending continuously on $\xi$. Hence, \cref{un.Geig} holds true for $v$ if it holds true for $v=0$ and $\sigma$ is sufficiently small. This together with part I finishes the proof of \cref{un.lem.gen} for a general $v$ if we take into account that the eigenvalues of $\widetilde{\mathcal G}_v^*(k)$ and $\widetilde{\mathcal G}_v(k)$ are related by complex conjugation.
\end{proof}

\begin{remark}\label{rem:injectivity_ass} It follows from analytic Fredholm theory (see, e.g., \cite[Cor. 3.3]{GS:71}) and 
\cref{re.lem.Fred} below that the condition that $\Re{\mathcal G}_{v_0}(k)$ be injective in $H^{-1/2}(\partial\Omega)$ is ``generically'' satisfied. More precisely, it is either satisfied for all but a discrete set 
of $k>0$ without accumulation points or it is violated for all $k>0$. Applying analytic Fredholm theory again 
to $z\mapsto \Re{\mathcal G}_{z^2v_0}(zk)$ and taking into account Proposition \ref{un.prp.hel}, 
we see that the latter case may at most occur for a discrete set of $z>0$ without accumulation points. 
\end{remark}

\begin{remark}\label{rem:injectivity_circ}
In the particular case of $v_0 = 0$, $\Omega = \{ x \in \mathbb R^d \colon |x| \leq R\}$, $d = 2$, $3$, the injectivity of $\Re {\mathcal G}_{v_0}(k)$ in $H^{-1/2}(\partial\Omega)$ is equivalent to the following finite number of inequalities:
\begin{equation}\label{un.jlyln0}
\begin{aligned}
   &j_l(kR) \neq 0 \text{ and } y_l(kR) \neq 0 && \text{for} \quad 0 \leq l < kR - \tfrac \pi 2, \quad d = 3, \\
   &J_l(kR) \neq 0 \text{ and }  Y_l(kR) \neq 0 &&\text{for} \quad 0 \leq l <  kR - \tfrac{\pi-1}{2}, \quad d = 2
\end{aligned}
\end{equation}
where $j_l$, $y_l$ are the spherical Bessel functions and $J_l$, $Y_l$ are the Bessel functions of integer order $l$. The reason is that the eigenvalues of $\Re {\mathcal G}_{v_0}(k)$ are explicitly computable in this case, see, e.g., \cite[p.104 \& p.144]{devaney:12}.
\end{remark}

\section{Conclusions}
In this paper we have presented, in particular, first local uniqueness results for inverse coefficient problems in wave equations with data given the imaginary part of Green's function on the boundary of a domain at a fixed frequency. In the case of local helioseismology it implies that small deviations of density and sound speed from the solar reference model are uniquely determined by correlation data of the solar surface within the given model.

The algebraic relations between the real and the imaginary part of Green's function established in this paper can probably be extended to other wave equations. An important limitation of the proposed technique, however, is that it is not applicable in the presence of absorption. 

To increase the relevance of uniqueness results as established in this paper to helioseismology and other applications, many of the improvements achieved for standard uniqueness results would be desirable: This includes stability results or even variational source conditions to account for errors in the model and the data, the use of many and higher wave numbers to increase stability, and results for data given only on part of the surface.

\bibliographystyle{siamplain}
\bibliography{green_art}
\end{document}